\documentclass[10pt]{article}

\usepackage{amsfonts,amssymb,amsthm,amsmath,mathtools,bbm}
\usepackage[OT1]{fontenc}
\usepackage[utf8]{inputenc}
\usepackage{authblk}
\usepackage[a4paper, total={6in, 9in}]{geometry}
\usepackage[colorlinks,citecolor=blue,urlcolor=blue]{hyperref}

\usepackage[usenames,dvipsnames,svgnames,table]{xcolor}
\newcommand{\revU}[1]{{#1}}
\newcommand{\revT}[1]{{#1}}
\newcommand{\revR}[1]{{#1}}
\newcommand{\revB}[1]{{#1}}

\newcommand{\pr}[1]{{\mathbb P}\left\{{\displaystyle #1}\right\}}
\newcommand{\expect}[1]{{\mathbb E}\left\{{\displaystyle#1}\right\}}

\newtheorem{theorem}{Theorem}
\numberwithin{theorem}{section}
\newtheorem{lemma}[theorem]{Lemma}
\newtheorem{definition}[theorem]{Definition}
\newtheorem{proposition}[theorem]{Proposition}
\newtheorem{corollary}[theorem]{Corollary}

\newcommand{\NatZero}{\mathbb{N}_0}

\newcommand{\vece}{\boldsymbol e}
\newcommand{\vecf}{\boldsymbol f}
\newcommand{\veca}{\boldsymbol a}
\newcommand{\vecg}{\boldsymbol g}

\newcommand{\vecx}{\boldsymbol x}
\newcommand{\vecd}{\boldsymbol d}
\newcommand{\vech}{\boldsymbol h}
\newcommand{\vecy}{\boldsymbol y}

\newcommand{\vecQ}{\boldsymbol Q}

\newcommand{\vecX}{\boldsymbol X}
\newcommand{\vecY}{\boldsymbol Y}

\newcommand{\vecI}{\boldsymbol I}
\newcommand{\vecA}{\boldsymbol A}
\newcommand{\vecD}{\boldsymbol D}
\newcommand{\vecm}{\boldsymbol m}
\newcommand{\vecq}{\boldsymbol q}
\newcommand{\tvecX}{\tilde{\vecX}}
\newcommand{\vecXI}{{\boldsymbol \Xi}}
\newcommand{\tX}{\tilde{X}}
\newcommand{\vecgamma}{\boldsymbol \gamma}

\newcommand{\vecalpha}{\boldsymbol \alpha}

\newcommand{\vecrho}{\boldsymbol \rho}
\newcommand{\veceta}{\boldsymbol \eta}
\newcommand{\vectheta}{\boldsymbol \theta}
\newcommand{\vecxi}{\boldsymbol \xi}

\newcommand{\veczero}{\boldsymbol 0}

\newcommand{\CC}{\mathcal{C}}
\newcommand{\cCC}{c \in \CC}
\newcommand{\prodc}{\prod_{\cCC}}
\newcommand{\sumc}{\sum_{\cCC}}

\newcommand{\Xck}{X_{c,n}}
\newcommand{\Ack}[1]{A_{ #1}}
\newcommand{\Dck}[1]{D_{#1}}
\newcommand{\Yc}{Y_c}
\newcommand{\Vc}{V_c}
\newcommand{\ZO}{{\mathbb N}_0}
\newcommand{\NN}{{\mathbb N}}

\newcommand{\LipIncr}{L_{\uparrow}}

\newcommand{\dtalp}{\dot{\alpha}}
\newcommand{\balp}{\boldsymbol{\alpha}}

\newcommand{\lb}{\left(}
\newcommand{\rb}{\right)}
\newcommand{\lc}{\left\{}
\newcommand{\rc}{\right\}}
\newcommand{\ls}{\left[}
\newcommand{\rs}{\right]}
\newcommand{\norm}[1]{\vert\vert #1 \vert\vert}
\newcommand{\abs}[1]{ \lvert #1 \rvert}
\newcommand{\Indic}[1]{{\mathbbm 1}\lc #1 \rc}
\newcommand{\Ominc}{\Omega_{-c}}
\newcommand{\Opluc}{\Omega_{+c}}
\newcommand{\pNa}{{\mathcal{N}}_{\text{A},c,n}}
\newcommand{\pNb}{{\mathcal{N}}_{\text{B},c,n}}
\newcommand{\pNt}{{\mathcal{N}}_{\text{T},c}}

\newcommand{\dtalpo}{\dot{\alpha}_\omega}
\newcommand{\tNa}{\tilde{\mathcal{N}}_{\text{A},c,n}}

\newcommand{\dPoiss}{d{\mathcal N}}
\newcommand{\dtPoiss}{d\tilde{\mathcal N}}
\newcommand{\dtNb}{\dtPoiss_{B,c,n}}
\newcommand{\dtNT}{\dtPoiss_{T_c}}

\title{Mean-field limits for large-scale random-access networks}
\author[1,4]{Fabio Cecchi \thanks{f.cecchi1@gmail.com}}
\author[1,2]{Sem C. Borst}
\author[1]{Johan S.H. van Leeuwaarden}
\author[3]{Philip A. Whiting}
\begin{small}
\affil[1]{Eindhoven University of Technology, Eindhoven, 5600 MB, Netherlands}
\affil[2]{Nokia Bell Labs, Murray Hill, NJ 07974, US}
\affil[3]{Macquarie University, North Ryde, NSW 2109, Australia}
\affil[4]{Bosch Center for AI, Sunnyvale, CA 94085, US}
\end{small}

\begin{document}

\maketitle

\begin{abstract}
We establish mean-field limits for large-scale random-access 
networks with buffer dynamics and arbitrary interference graphs.
While saturated-buffer scenarios have been widely investigated
and yield useful throughput estimates for persistent sessions,
they fail to capture the fluctuations in buffer contents over time,
and provide no insight in the delay performance of flows with
intermittent packet arrivals.
Motivated by that issue, we explore in the present paper random-access
networks with buffer dynamics, where flows with empty buffers refrain
from competition for the medium.
The occurrence of empty buffers thus results in a complex dynamic
interaction between activity states and buffer contents,
which severely complicates the performance analysis.
Hence we focus on a many-sources regime where the total number
of nodes grows large, which not only offers mathematical tractability
but is also highly relevant with the densification of wireless
networks as the Internet of Things emerges.
We exploit time scale separation properties to prove that the
properly scaled buffer occupancy process converges to the solution
of a deterministic initial-value problem, and establish the
existence and uniqueness of the associated fixed point.
This approach simplifies the performance analysis of networks with
huge numbers of nodes to a low-dimensional fixed-point calculation.
For the case of a complete interference graph, we demonstrate
asymptotic stability, provide a simple closed-form expression
for the fixed point, and prove interchange of the mean-field
and steady-state limits.
This yields asymptotically exact approximations for key performance
metrics, in particular the stationary buffer content and packet
delay distributions.
\end{abstract}

\footnotemark{\footnotesize{This work was done while F.Cecchi was at Eindhoven University of Technology.}}

\section{Introduction}
\label{Sec_Intro}

\subsection{Background and related work}
\label{Sec_Intro_BRW}

Wireless networks are already large and complex today, and being
at the heart of the so-called Internet of Things (IoT)~\cite{AIM10},
are expected to grow even denser in the future~\cite{Evans11}. 
Obviously, when the number of nodes is large, in the  hundreds or even
thousands of nodes, a dedicated medium or channel cannot be assigned
to each node, and nodes have to share the medium. 
Medium access control (MAC) mechanisms are \revR{therefore crucial
to resolve the contention among the various nodes}.
However, in large networks, a centralized control mechanism is hard
to implement and to maintain since it would require constant status
updates generating prohibitive communication overhead.
For this reason the design of efficient distributed (local) MAC
protocols has attracted a lot of attention.

A very popular distributed MAC mechanism is the CSMA (Carrier-Sense
Multiple-Access) protocol, which is currently at the core of the IEEE
802.11 and 802.15.4 standards.
Its popularity is mostly due to its simplicity and efficiency.
The key feature of the CSMA protocol is that each node waits 
for a random back-off period before initiating a transmission.
Interference is avoided since the back-off countdown is interrupted
whenever potential interference is sensed, and only resumed once the
medium is sensed idle again.
This protocol, whilst extremely easy to understand on a local level,
generates complex and interesting macroscopic network dynamics.

In the performance analysis of CSMA networks, a common assumption is
the existence of an underlying graph that represents interference
between the various nodes in the network. 
An edge between two nodes means that destructive interference is
caused by simultaneous transmission.
Both empirical and theoretical support for the notion of 
an interference graph is provided in \cite{HT15,ZZWYZ13}.

When the nodes always have packets to transmit, the network is said
to be {\it saturated} and the macroscopic activity behavior is amenable
to analysis under the assumption of an interference graph.
In particular, the activity process has an elegant product-form
stationary distribution \cite{BKMS87,LKLW10,VBLP10}.
The computation of the stationary distribution of the activity process
reduces to the identification of all the subsets of nodes which may
transmit simultaneously, namely the independent sets of the
interference graph.

Real-life scenarios however involve {\it unsaturated} networks.
Packets arrive at the various nodes according to exogenous processes,
and buffers may drain from time to time as packets are transmitted.
In particular, in IoT applications, sources are likely to generate
packets only sporadically, with fairly tight delay constraints,
and often have empty buffers.
Since empty nodes temporarily refrain from the medium competition, the
activity process is strictly intertwined with the buffer content process.
In this situation, the product-form solution no longer holds
\cite{CBL14,VBLP10} and an exact stationary analysis does not seem
tractable. 

The analysis of unsaturated CSMA networks simplifies if certain
symmetry conditions amongst the various nodes hold.
An important instance is when there is a substantial number of nodes
with similar traffic and placement in the network, so that the
operation of one is equivalent to that of many others.
More generally, nodes can be divided into classes with the symmetry
conditions now applying to nodes of the same class.
The asymptotic regime where the number of such nodes in each class
grows to infinity, is commonly referred to as a {\it mean-field regime}.
Mean-field theory originated in physics, where it is still widely used
in analyzing models involving a large number of interacting particles. 
The aggregated effect of all the other nodes on any tagged node is
approximated by a single averaged effect (the mean-field),
thus reducing a many-body problem to a more tractable one-body problem. 
In the context of random-access networks, a mean-field regime not only
provides analytical tractability, but is also highly relevant in the
context of the envisioned massive numbers of IoT devices. 

A thorough survey of mean-field analysis of random-access protocols 
is presented in~\cite{Duffy10}.
The work of Bianchi~\cite{Bianchi00} is a landmark paper which assumed
\revR{nodes to behave independently one from the other in the regime
where many of them are present} so as to derive tractable formulae
for the key performance measures of the system.
The papers surveyed in~\cite{Duffy10} mostly use mean-field theory
to provide either evidence or objection for the assumption of Bianchi.
Among these papers, it is worth mentioning~\cite{CLJ12}, where the
authors investigated the existence of a global attractor for the
mean-field system and provided sufficient conditions for its existence, 
deducing the validity of Bianchi's assumption.
Further papers which deserve to be mentioned are \cite{BMP08,SGK09},
where the authors exploited mean-field theory so as to obtain
approximations for key performance measures of large systems. 
In particular, \cite{BMP08} focuses on the characterization of the
stability region, while \cite{SGK09} examines the throughput
performance of the system.
None of the above-mentioned papers considered scenarios
with unsaturated buffers, with the exception of~\cite{BMP08},
which however dealt with systems evolving in discrete time and did not
consider performance metrics like packet delays. 

\subsection{Key contributions and paper organization}
\label{Sec_Intro_Key} 
 
In the present paper we examine the buffer dynamics in large-scale
unsaturated random-access networks.
Specifically, we analyze the buffer occupancy processes in a mean-field
regime where the number of nodes grows large.

\revU{We provide a detailed model description and introduce some
useful notation and preliminaries in Section~\ref{Sec_Model}.
An overview of the main results of the paper is presented
in Section~\ref{Sec_Ove}.
In Section~\ref{Sec_Gen} we prove for general interference graphs
that a suitably scaled version of the buffer occupancy processes converges
in the mean-field limit to a tractable deterministic initial-value problem.
We also establish necessary and sufficient conditions for the existence
and uniqueness of a fixed point of the initial-value problem,
and provide a characterization  of the fixed point as the solution
of a low-dimensional equation.
In Section~\ref{Sec_Compl} we focus on scenarios with a complete
interference graph, and demonstrate global asymptotic stability of the
initial-value problem, i.e., convergence to the unique fixed point
from any initial state with finite mass.
We then proceed to show positive recurrence of the pre-limit process
and tightness of the sequence of stationary distributions,
and combine these properties to prove interchange of the mean-field
and steady-state limits.
The interchange of limits is leveraged to establish that the
stationary buffer content distributions at the various nodes converge
to geometric distributions, while the stationary distributions of the
scaled waiting time and sojourn time converge to exponential distributions.
The parameters of these limiting distributions are directly expressed
in terms of the fixed point of the initial-value problem.
These results provide asymptotically exact approximations for the
stationary waiting-time and sojourn time distributions.
In Section~\ref{Sec_Num} we present some simulation experiments
to illustrate the analytical results. 
Finally, in Section~\ref{Sec_Conclusion} we make a few concluding
remarks and offer several suggestions for further research.}
 
\section{Model description}
\label{Sec_Model}

We consider a network of $N$~nodes sharing a wireless medium according
to a random-access protocol.
The various nodes are grouped into a set of classes/clusters
$\mathcal{C} = \{1,\dots,C\}$ such that nodes in the same class have
the same statistical characteristics.
Denote by $N^{(N)}_c$ the number of class-$c$ nodes,
where $\sumc N^{(N)}_c = N$ and $ p_{c,N} = N^{(N)}_c/N$. \\

\noindent
{\it Interference graph}.
Given a class-wise interference graph $G = (\CC, \mathcal{E})$,
two nodes interfere when they belong either to the same class
or to two neighboring classes in~$G$.
A feasible class activity state can thus be represented by a vector
$\omega \in \{0, 1\}^C$, with $\omega_c = 1$ if a class-$c$ node
is transmitting in state~$\omega$ and $\omega_c = 0$ otherwise,
and $\omega_c\omega_d =0$ if $(c, d) \in \mathcal{E}$.
Let $\Omega$ be the set of all feasible class activity states, which are
in one-to-one correspondence with the independent sets of the graph~$G$.
For every class~$c$, we define the following subsets of~$\Omega$:
\begin{align*}
&\Omega_{-c} \:\dot=\: \{\omega \in \Omega: \omega_c = 0, \omega_d = 0
\:\forall\: d \text{ s.t. } (c,d) \in \mathcal{E}\}, \\
&\Omega_{+c} \:\dot=\: \{\omega \in \Omega: \omega_c = 1\}.
\end{align*}
This means that $\omega \in \Omega_{-c}$ if and only if in the class
activity state~$\omega$ none of the nodes belonging to class~$c$
or to a class interfering with class~$c$ are active,
while $\omega \in \Omega_{+c}$ if and only if a class-$c$ node is active.
We define the class-capacity region of the network as
\revT{$\Gamma \doteq \text{Conv}(\Omega)$, i.e.,}
\[
\Gamma = \{\vecgamma \in \mathbb{R}_+^C: \exists\:
\veca=(a_\omega)_{\omega \in \Omega} \:\text{ s.t. }\:
a_\omega \geq 0, \sum_{\omega \in \Omega} a_\omega \leq 1,
\sum_{\omega \in \Omega} a_\omega \omega = \vecgamma\}.
\]

As an illustration, Figure~\ref{fig_SquareGraph} shows a square
interference graph~$G$ with $C = 4$ classes of nodes \revU{numbered
in a clockwise fashion}.
The activity states are
\begin{align*}
\Omega = \big\{& (0,0,0,0), (1,0,0,0), (0,1,0,0), (0,0,1,0), \\
& (0,0,0,1), (1,0,1,0), (0,1,0,1) \big\},
\end{align*}
and the class-capacity region is given by 
\[
\Gamma = \{\vecgamma \in \mathbb{R}_+^4:
\max\{\gamma_1,\gamma_3\} + \max\{\gamma_2,\gamma_4\} \leq 1\}.
\]
Taking $c=1$ the sets $\Ominc,\Opluc \subset \Omega$ are
\begin{eqnarray*}
\Omega_{-1} & = & \lc (0,0,0,0), (0,0,1,0) \rc \\
\Omega_{+1} & = & \lc (1,0,0,0), (1,0,1,0) \rc.
\end{eqnarray*}
\begin{figure}
\begin{center}
\includegraphics[width=0.7\textwidth]{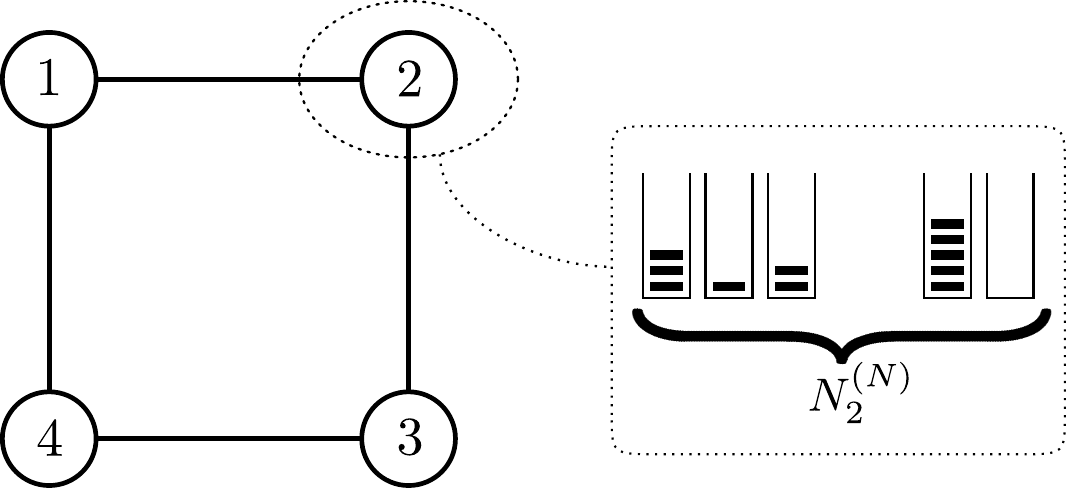}
\caption{A square interference graph with $C = 4$.
\label{fig_SquareGraph}}
\end{center}
\end{figure}

\noindent
{\it Network dynamics}.
\revT{Nodes within the same class share the same statistical features.}
Packets arrive at the various class-$c$ nodes as independent Poisson
processes of rate $\lambda_c^{(N)} \:\dot=\: \lambda_c / N^{(N)}_c$.
Once a class-$c$ node obtains access to the medium, it transmits one
packet, which takes an exponentially distributed time with parameter
$\mu_c^{(N)} \:\dot=\: \mu_c$.
In between two consecutive transmissions a class-$c$ node must
back-off for an exponentially distributed time period with parameter
$\nu_c^{(N)} \:\dot=\: \nu_c / N^{(N)}_c$.
Note that a node refrains from the back-off competition whenever its
buffer is empty.
Also, the back-off period of a class-$c$ node is suspended (frozen)
when the medium is occupied by an interfering node, i.e., the class
activity state does not belong to $\Omega_{-c}$. 

Define the queue length process
$\vecQ^{(N)}(t) = (Q_{c,k}^{(N)}(t))_{\cCC, k = 1,\ldots,N_c}$,
where $Q_{c,k}^{(N)}(t)$ represents the number of packets in the
buffer of the $k$-th node belonging to class~$c$ at time~$t$ 
excluding any possible packet in transmission, and define
by $\vecY^{(N)}(t)$ the class activity process on the state space $\Omega$.
Observe that $\big(\vecQ^{(N)}(t), \vecY^{(N)}(t)\big)$ evolves
as a Markov process. \\

\noindent
{\it A population process description}.
To ease the notation we write $N_c = N^{(N)}_c$ whenever possible.
Due to the class symmetry, the nodes within a class are statistically
indistinguishable, and the system state may be described in terms
of the numbers of nodes belonging to the same class and with the same
number of packets in the buffer.
In particular, define the population process
$\tvecX^{(N)}(t) = (\tX_{c,n}^{(N)}(t))_{\cCC,n\in\NatZero}$, where
\[
\tX_{c,n}^{(N)}(t) =
\frac{1}{N_c} \sum_{k = 1}^{N_c} \mathbbm{1}\{Q_{c,k}^{(N)}(t)=n\}.
\]
The process $\big(\tvecX^{(N)}(t),\vecY^{(N)}(t)\big)$ is itself
Markovian with state space $E \times \Omega$, where 
\begin{equation}
E \:\dot=\: \chi^C, \qquad \chi \:\dot=\: \{\vecx = (x_n)_{n\in\NatZero}:
\sum_{n\in\NatZero} x_n \leq 1,\: x_{n} \geq 0, \: \forall n\in\NatZero\}.
\label{eqn_Edefn}
\end{equation}
Observe that the population process lies in $E^1\subseteq E$ where
\begin{equation}
E^1 \:\dot=\:
\{\vecx \in E \:: \sum_{n=0}^{\infty} x_{c,n} = 1, \: \forall c \in\CC\}.
\label{eqn_E1defn}
\end{equation}

The possible transitions for the process
$\big(\tvecX^{(N)}(t),\vecY^{(N)}(t)\big)$ may be described as follows:

\textbullet\: A packet arrives at a class-$c$ node having $n$~packets
in its buffer:
this happens at rate $\lambda_c^{(N)}$ times the number of class-$c$
nodes in state~$n$, i.e.,
$\lambda_c^{(N)} N_c \tX^{(N)}_{c,n} = \lambda_c \tX^{(N)}_{c,n}$,
and generates the transition
\[
\begin{cases}
\tX_{c,n}^{(N)} \rightarrow \tX_{c,n}^{(N)} - \frac{1}{N_c} \\
\tX_{c,n+1}^{(N)} \rightarrow \tX_{c,n+1}^{(N)} + \frac{1}{N_c},
\end{cases}
\]
i.e.,
\[
(\tvecX^{(N)},\vecY^{(N)}) \rightarrow (\tvecX^{(N)} +
\frac{1}{N_c}\vece^{c,n+1}_{c,n},\vecY^{(N)}),
\]
\revR{where $\vece^{c_1,n_1}_{c_2,n_2} = \vece_{c_1,n_1} - \vece_{c_2,n_2}$
and $\vece_{c,n} \in \chi^{C}$ has all $0$ entries except a $1$
in position $(c,n)$.}

\textbullet\: A transmission is completed by a class-$c$ node:
this happens at rate $\mu_c^{(N)}= \mu_c$ and only if a class-$c$ node
is transmitting, i.e., $\vecY^{(N)} \in \Omega_{+c}$.
The transition generated is the following
\[
(\tvecX^{(N)},\vecY^{(N)}) \rightarrow (\tvecX^{(N)},\vecY^{(N)}-\vece_c),
\]
\revR{where $\vece_c \in \{0,1\}^C$ has all $0$ entries except a $1$
in position $c$.}
 
\textbullet\: A back-off is completed by a class-$c$ node having
$n > 0$ packets in its buffer:
this happens only if the class activity state allows class-$c$ nodes
to back-off, i.e., $\vecY^{(N)} \in \Omega_{-c}$, and at rate
$\nu_c^{(N)}$ times the number of class-$c$ nodes in state~$n$, i.e., 
$\nu_c^{(N)} N_c \tX^{(N)}_{c,n} =  \nu_c \tX^{(N)}_{c,n}$,
and generates the transition
\[
\begin{cases}
\vecY^{(N)} \rightarrow \vecY^{(N)} + \vece_c \\
\tX_{c,n}^{(N)} \rightarrow \tX_{c,n}^{(N)} - \frac{1}{N_c} \\
\tX_{c,n-1}^{(N)} \rightarrow \tX_{c,n-1}^{(N)} + \frac{1}{N_c},
\end{cases}
\]
i.e.,
\begin{small}
\[
(\tvecX^{(N)},\vecY^{(N)}) \rightarrow
(\tvecX^{(N)} + \frac{1}{N_c} \vece^{c,n-1}_{c,n}, \vecY^{(N)} + \vece_c).
\]
\end{small}

\noindent
{\it Preliminary results for saturated scenario}.
As mentioned earlier, we focus on an unsaturated network where nodes
with empty buffers refrain from competition for the medium.
For later purposes, it is convenient to also consider a related
saturated network where each class behaves as a single node {\it always} 
competing for the medium, and backing-off at rate~$\eta_c\nu_c$,
where $\veceta = (\eta_c)_{\cCC} \in \mathbb{R}_+^C$.
The transmission times of the node associated with class~$c$ are
exponentially distributed with parameter~$\mu_c$.

For compactness, denote $\sigma_c = \nu_c / \mu_c$, and for any
$\veceta \in \mathbb{R}_+^C$,
\[
Z(\omega; \veceta) \:\dot=\: \prodc (\eta_c \sigma_c)^{\omega_c},
\qquad \omega \in \Omega,
\]
and
\begin{equation*}
Z(\veceta) \:\dot=\: \sum_{\omega \in \Omega} Z(\omega; \veceta).
\end{equation*}
The activity process in this fictitious saturated network has
a product-form stationary distribution \cite{BKMS87,VBLP10,WK05}
\begin{equation}
\label{Eqn_piDef}
\pi(\omega; \veceta) =
\frac{Z(\omega; \veceta)}{Z(\veceta)},
\qquad \omega \in \Omega.
\end{equation}

Define now the throughput function
$\vectheta: \mathbb{R}_+^C \rightarrow \Gamma$, where
\begin{equation}
\label{Eqn_throuDef}
\theta_c(\veceta) \:\dot=\: \sum_{\omega \in \Omega} \omega_c \pi(\omega; \veceta) =
\sum_{\omega \in \Omega_{+c}} \pi(\omega; \veceta) = \pi(\Omega_{+c}; \veceta) 
\end{equation}
represents the fraction of time that the node associated with class~$c$
is active in the saturated network \revU{in stationarity}.

As proved in \cite{JW08,VJLB11}, the throughput map~$\vectheta$ is
globally invertible, i.e., for any achievable throughput vector
$\vecgamma \in \text{int}(\Gamma)$ there exists a unique vector 
$\veceta(\vecgamma)\in{\mathbb R}_+^C$ such that
\begin{equation}
\label{Eqn_InverseDef}
\vectheta(\veceta(\vecgamma)) = \vecgamma.
\end{equation}
Intuitively, the $c$-th coordinate of $\veceta(\vecgamma)$ represents
by what factor the back-off rate at node $c$ needs to
accelerate/decelerate in order for the target throughput
vector~$\vecgamma$ to be achieved in stationarity.

\section{Overview of the main results}
\label{Sec_Ove}

In this section we provide an overview of the main results of the paper, 
along with an interpretation and high-level discussion of their
ramifications, before presenting the proofs in the subsequent sections.

\revU{
We first analyze networks with general class-based interference graphs,
and define
\[
\vecX^{(N)}(t) \doteq \tvecX^{(N)}(Nt)
\]
as the {\it fluid version\/} of the population process.
We consider a sequence of processes $\vecX^{(N)}(t)$ as the number
of nodes in the system increases, and establish its weak convergence
to $\vecx(t)$, the solution of a tractable deterministic initial-value
problem as stated in the next theorem.
}

\begin{theorem}
\label{Theo_MFGen}
Assume $\vecX^{(N)}(0) \xrightarrow{N\rightarrow\infty} \vecx^{\infty}\in E^1$
and $\lim_{N\rightarrow\infty}p_{c,N} = p_c > 0$ for every
$c\in\mathcal{C}$.
Then the sequence of processes 
\[
(\vecX^{(N)}(t))\in D_{E^1}[0,\infty)
\]
has a continuous limit $\vecx(t)\in C_{E^1}[0,\infty)$ which is
determined by the unique solution of the initial-value problem
\begin{equation}
\label{eqn_MFDiffEqnGen}
\frac{d \vecx(t)}{dt} = H(\vecx(t)), \qquad
\vecx(0) = \vecx^{\infty},
\end{equation}
where the function $H(\cdot)$ is defined by
\begin{align*}
 H(\vecx) = & \sumc \frac{1}{p_c} \Big(\lambda_c \sum_{n=0}^{\infty} x_{c,n}\vece^{c,n+1}_{c,n} + \pi_{\vecx_0}(\Omega_{-c}) \nu_c \sum_{n=1}^{\infty} x_{c,n}\vece^{c,n-1}_{c,n}\Big)
\end{align*}
\revR{with \[
\pi_{\vecx_0}(\Omega_{-c}) = \sum_{\omega \in \Omega_{-c}} \pi_{\vecx_0}(\omega), \qquad \pi_{\vecx_0}(\omega)\doteq \pi(\omega; \vece - \vecx_0),
\]
$\vecx_0 = (x_{1,0},\ldots,x_{C,0})$, and $\vece$ has every component
equal to~$1$.}
\end{theorem}

\noindent
The major difficulty in the analysis of unsaturated networks arises
from the correlation between the population process $\tvecX^{(N)}(t)$
and the activity process $\vecY^{(N)}(t)$. 
A key observation in the proof of Theorem~\ref{Theo_MFGen} is that
these processes evolve on different time scales, i.e., the population
process evolves $N$~times slower than the activity process. 
Hence, in the mean-field regime, the rapidly changing activity process
``converges to an instantaneous measure on the activity states"
determined by the current fraction of queues which are empty.
This is the reason why the initial-value problem~\eqref{eqn_MFDiffEqnGen}
does not explicitly involve the activity process and the evolution
of the population process at time~$t$ depends only on $\vecx(t)$.  
Specifically, the fraction of class-$c$ nodes with $n \geq 1$ packets
in the buffer ``increases" at rate $\lambda_c$ and ``decreases" at rate
$\nu_c\pi_{\vecx_0}(\Omega_{-c})$, where the measure
$\pi_{\vecx_0}(\Omega_{-c})$ represents the limiting ``instantaneous
measure" on the activity states that allows class-$c$ nodes to back-off.
The argument is thus \revR{based on a stochastic averaging principle
which follows the same lines of ideas of \cite{FR14,HK94}}. \\
 

The initial-value problem~\eqref{eqn_MFDiffEqnGen} is certainly easier
to analyze than the population process in a network with a finite
number of nodes. 
Theorem~\ref{Theo_EquilibChar} states that, under certain necessary
and sufficient conditions, there exists a unique \revT{fixed}
point~$\vecx^*$ for the initial-value problem~\eqref{eqn_MFDiffEqnGen}, 
and provides its characterization in terms of the load of the network.

\begin{theorem}
\label{Theo_EquilibChar}
If $\vecrho \in \mbox{int}(\Gamma)$
and $\vecxi = \veceta(\vecrho) < \vece$,
where $\veceta(\cdot)$ is defined in~\eqref{Eqn_InverseDef},
then $\vecx^* \in E^1$, with
\begin{equation}
\label{EquilibriumRelation}
x_{c,n}^* = (1 - \xi_c) \xi_c^n,
\end{equation}
is the unique \revT{fixed} point of the initial-value
problem~\eqref{eqn_MFDiffEqnGen} in~$E^1$, i.e., $H(\vecx^*) = \veczero$.
\end{theorem}

\noindent
Specifically, the condition $\vecrho \in \mbox{int}(\Gamma)$ is
necessary for existence and ensures the load~$\vecrho$ of the system
is sustainable.
On the other hand, the condition $\vecxi  = \veceta(\vecrho) < \vece$
ensures that the desired throughput vector~$\vecrho$ can be achieved
by a feasible back-off vector.
Although Theorem~\ref{Theo_EquilibChar} only concerns the scaled
population process, combined with Theorem~\ref{Theo_MFGen} it yields
that the stationary distribution of the associated class activity
process is given by $\pi_{\vecx_0^*}\in \mathbb{P}(\Omega)$ in the limit.
In other words, in the limit the stationary distribution of the class
activity process is the same as in a scenario where the aggregate
back-off rate of class~$c$ is a {\it constant fraction}~$\xi_c$ of the
nominal back-off rate $\nu_c$.
This is consistent with the fact that $\xi_c$ is the stationary
fraction of class-$c$ nodes that have non-empty buffers and compete
for the medium, and $\vecxi$ will therefore in the sequel be referred
to as the vector of {\it activity factors}.
The equation $\vecxi = \veceta(\vecrho)$ reflects that the activity
factors must be such that each class~$c$ is active a fraction~$\rho_c$
of the time. \\
 
\begin{figure*}[t]
\centering
\includegraphics[width=0.75\textwidth]{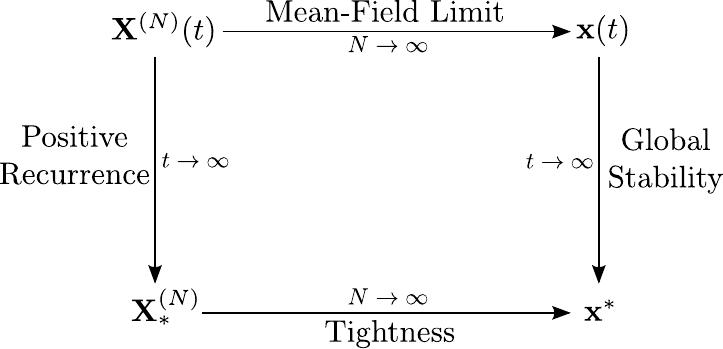}
\caption{Diagram for the interchange of limits.}
\label{FigSquare}
\end{figure*}  

We will use the above convergence properties of the population process
to derive asymptotically exact results for the key performance measures
of the system.
For a rigorous treatment, we focus on the scenario where all the
classes mutually interfere, i.e., the interference graph
$G = (\mathcal{C},\mathcal{E})$ is complete, in which case 
\[
\Gamma = \{\vecgamma \in \mathbb{R}^C: \sum_{c \in \CC} \gamma_c \leq 1\}.
 \]
In this framework, the nodes of every class are allowed to back-off
only when the activity process is in the idle state.
We will show that the vector of activity factors~$\vecxi$
in Theorem~\ref{Theo_EquilibChar} simplifies to
\[
\xi_c = \frac{\lambda_c}{\nu_c \big(1 - \sum_{d \in \CC} \frac{\lambda_d}{\mu_d}\big)},
\]
so that the condition $\vecxi  = \veceta(\vecrho) < \vece$ can be
expressed in explicit form as
\begin{equation}
\label{eqn_Gen_Assum} 
\max_{c \in \CC}\frac{\lambda_c}{\nu_c} < 1 - \sum_{d \in \CC} \frac{\lambda_d}{\mu_d}, 
\end{equation}
which forces the right-hand side to be positive, i.e.,
$\vecrho \in \mbox{int}(\Gamma)$.
When condition~\eqref{eqn_Gen_Assum} applies, Theorem~\ref{Theo_Comp_IL}
holds and asymptotically characterizes the stationary distribution
of the population process as the \revT{fixed} point of the initial-value
problem~\eqref{eqn_MFDiffEqnGen}.  

\begin{theorem}
\label{Theo_Comp_IL}
The sequence of stationary \revU{random variables
$(\vecX_*^{(N)})_{N \geq 1}$ weakly converges} to ${\vecx^*} \in E^1$,
i.e.,
\begin{equation}
\label{eqn_Com_Objective2}
\vecX_*^{(N)} \:\: \xRightarrow{N\rightarrow\infty} \:\:\vecx^*. 
\end{equation}
\end{theorem}

\noindent
In order to prove Theorem~\ref{Theo_Comp_IL}, we show that the
interchange of limits displayed in Figure~\ref{FigSquare} holds. 
Specifically, \revR{the methodology developed} reduces the derivation
of the stationary distribution of an intractable $N$-dimensional
Markov process to the computation of the \revT{fixed} point
of a low-dimensional fixed-point equation. 
 
 
Theorem~\ref{Theo_Comp_IL} is exploited so as to obtain approximations
for the performance measures of the system.
In particular, denote by $Q_c^{(N)}$ the stationary queue length
at a class-$c$ node, and by $W_c^{(N)}$ and $S_c^{(N)}$ the stationary
waiting time and sojourn time of a packet at a class-$c$ node,
respectively.
Assume condition~\eqref{eqn_Gen_Assum} holds.

\begin{theorem}
\label{Theo_QWSNcWeakConvergence}
For every class~$c$,
\[
Q_c^{(N)} \Rightarrow \bar{Q}_c, \qquad
\frac{\lambda_c}{N_c} S_{c}^{(N)} \Rightarrow \bar{S}_c, \qquad
\frac{\lambda_c}{N_c} W_{c}^{(N)} \Rightarrow \bar{W}_c,
 \]
where {\normalfont 
\[
\bar{Q}_c \sim \text{Geo}(\xi_c), \qquad \bar{W}_c \sim \bar{S}_c \sim
\text{Exp}\big(\frac{1-\xi_c}{\xi_c}\big),\qquad c \in \mathcal{C}.
 \]}
\end{theorem}

\begin{theorem}
\label{Theo_QWSNcExpConvergence}
For every class~$c$, 
\[
\mathbb{E}[ Q_c^{(N)}] \rightarrow \mathbb{E}[\bar Q_c], \qquad
\frac{\lambda_c}{N_c}\mathbb{E}[W_c^{(N)}] \rightarrow \mathbb{E}[\bar W_c], \qquad  
\frac{\lambda_c}{N_c}\mathbb{E}[S_c^{(N)}] \rightarrow \mathbb{E}[\bar S_c].   
\]
\end{theorem}

\noindent
These theorems show that, given any node in the system, its stationary
queue length converges in distribution and expectation to a geometric
random variable, while the scaled waiting time and sojourn time of a packet
in its buffer converge to exponentially distributed random variables.
These limits can be used as approximations for the performance
of finite-node systems, which are provably exact as the number of nodes
in the network grows large, and remain very accurate when the size of the network is moderate \cite[Section 4.6]{Cecchi18}.

\section{Mean-field limit in a general framework}
\label{Sec_Gen}

In this section we present the proofs of Theorems~\ref{Theo_MFGen}
and~\ref{Theo_EquilibChar} as stated in Section~\ref{Sec_Ove}.
 
\subsection{Derivation of the mean-field limit}
\label{Sec_Gen_FL}
 
Theorem~\ref{Theo_MFGen} describes the weak limit of the sequence
of processes $\{\vecX^{(N)}(t)\}_{N\in\mathbb{N}}$, which is known as
the mean-field limit of the population process.
In order to prove Theorem~\ref{Theo_MFGen}, we need to go through
various steps which we briefly outline here. 
\begin{enumerate}
\item In Section~\ref{Sec_Gen_FL_UnitPoisson}, a Poisson representation
for the prelimit process
\[
\big(\tvecX^{(N)}(t),\vecY^{(N)}(t)\big)
\]
is provided.
This representation allows us to easily derive the dynamics of the
process on a fluid time scale.
\item In Section~\ref{Sec_Gen_FL_YRand} a detour is needed in order
to ensure the existence of the limiting process. 
Specifically, the class activity process $\vecY^{(N)}$ is replaced
by a cumulative time 
process $\vecalpha^{(N)}$, so that the 
evolution of the model on fluid time scale can be equivalently
described via the process 
\begin{equation}
\label{eqn_XalphaPrelimitProcess}
\big(\vecX^{(N)}(t),\vecalpha^{(N)}(t)\big).
\end{equation}
Finally, the sequence of processes~\eqref{eqn_XalphaPrelimitProcess}
is shown to weakly converge as $N \rightarrow \infty$ to a limiting process
\begin{equation}
\label{eqn_XalphaLimitProcess}
\big(\vecx(t),\vecalpha(t)\big).
\end{equation}
\item In Section~\ref{Sec_Gen_FL_FL} we characterize the limiting
process~\eqref{eqn_XalphaLimitProcess}.
\item As a last step, in Section~\ref{Sec_Gen_FL_WeakLimalpha},
we describe how the limiting process $\vecalpha(t)$ at time~$t$ is
uniquely determined by the value $\vecx(t)$. 
Hence a self-contained deterministic initial-value problem governing
the behavior of $\vecx(t)$ is obtained.
\end{enumerate}
 
In preparation for the analysis, we briefly recall a few useful
properties of the topological space where the population process evolves.
The sample paths of $\vecX^{(N)}(t)$ lie in $D_E[0,\infty)$, i.e.,
the set of the cadlag functions from $[0,\infty)$ in $E = \chi^C$,
where the space~$E$ is defined in~\eqref{eqn_Edefn}.  
Define on~$E$ the metric $\rho(\cdot,\cdot)$ such that
\[
\rho\big((\vecx^{(1)}_c)_{\cCC},(\vecx^{(2)}_c)_{\cCC}\big) =
\sumc \rho_{1} (\vecx^{(1)}_c,\vecx^{(2)}_c),
\]
where $\vecx_c = (x_{c,n})_{n \in \NatZero} \in \chi$ and $\rho_{1}$
is a metric on $\mathbb{R}^{\infty}$, i.e.,
\[
\rho_1(\vecx^{(1)},\vecx^{(2)}) = \sum_{n\in\NatZero} 2^{-n}
\frac{|x_n^{(1)}-x_n^{(2)}|}{1 + |x_n^{(1)}-x_n^{(2)}|}.
\]
Note that $(\mathbb{R}^{\infty}, \rho_1)$ is separable and complete,
and $\chi$ is compact in $(\mathbb{R}^{\infty}, \rho_1)$ as each
coordinate lies in $[0,1]$, see \cite[page 219]{Billingsley68}.
From \cite[Section M6, pages 240-241]{Billingsley99}, the following
lemma is therefore obtained.

\begin{lemma}
\label{ECompact}
The subset $E = \chi^{C}$ is complete, separable, and compact
under the product topology induced by the metric~$\rho$ in the
space~$\prodc \mathbb{R}^{\infty}$.
\end{lemma}

\noindent 
In the subsequent analysis we will work with $D_{E}[0,\infty)$ using
the metric~$d$, an exponentially weighted version of the Skorohod metric,
see~\cite[page 117]{EK86}.
Under this metric $D_E[0,\infty)$ is complete and separable as $(E,\rho)$
is itself complete and separable, see~\cite[Theorem 5.6, page 121]{EK86}.

\subsubsection{Unit Poisson process representation}
\label{Sec_Gen_FL_UnitPoisson}

In this subsection we describe the fluid-time marginals of the
prelimit Markov population process $\big(\tvecX^{(N)}(t),\vecY^{(N)}(t)\big)$
in terms of infinite sequences of unit Poisson processes. 
For each $c\in \CC$, and then for each $n \in \ZO$, the random quantity
$\pNa(t)$ determines the class-$c$ arrivals to queues with $n$~packets
in the buffer during the time interval $[0,t)$.
We similarly define $\pNb(t)$ for each $c \in \CC$, $n \in \NN$,
to determine the back-offs and finally $\pNt(t)$ determines the
process of transmissions for each $c \in \CC$.
These processes are supposed mutually independent and defined on a common
probability space $\lb \Omega_F, \mathcal{F}^F, \mathbb{P} \rb$.
We also define the centered versions of these processes to be
$\tNa(s) = \pNa(s) - s$ for each $c \in \CC$, $n \in \ZO$,
with corresponding definitions for the back-offs and the transmissions.
These are martingales with respect to their natural filtrations.

To obtain the Poisson process representation, 
we first fix the initial conditions, which are deterministic.
Realizations of the above unit Poisson processes are then drawn
and the (unscaled) population random variables $\tvecX$, $\vecY$ are
obtained as solutions to the following equations, where fluid time is~$t$.
That is,
\begin{align*}
\tvecX^{(N)}(t) =\:&\tvecX^{(N)}(0) + \sumc \sum_{n=0}^{\infty} \frac{\vece^{c,n+1}_{c,n}}{N_c^{(N)}}\tilde A_{c,n}(t) + \sumc\sum_{n=1}^{\infty} \frac{\vece^{c,n-1}_{c,n}}{N_c^{(N)}}\tilde D_{c,n}(t) 
\end{align*}
\begin{align*}
\vecY^{(N)}(t)=\:&\vecY^{(N)}(0)  - \sum_{c\in \CC} \vece_c \tilde V_{c}(t) + \sum_{c\in \CC} \sum_{n=1}^\infty \vece_c \tilde D_{c,n}(t),
\end{align*}
where
\begin{align*}
\tilde A_{c,n}(t) & =  \pNa\Bigl( \lambda^{(N)}_c \int_0^t N_c^{(N)}\tilde X_{c,n} ^{(N)}(s) ds \Bigr),  \\
\tilde D_{c,n}(t) & =  \int_0^{t} \mathbbm{1}\{\vecY^{(N)}(s^-)\in \Ominc\} \dPoiss_{\text{B},c,n}\Bigl( \nu_c^{(N)} \int_0^s N^{(N)}_c \tX^{(N)}_{c,n}(u) du \Bigr), \\
\tilde V_{c}(t) & =   \int_0^{t} \mathbbm{1}\{\vecY^{(N)}(s^-)\in \Opluc\} \dPoiss_{\text{T},c}\Bigl( \mu_c^{(N)} \int_0^s  du \Bigr).
\end{align*}
As observed in the discussion in~\revU{\cite{Kurtz80}}, the solutions
to these equations can be obtained by fixing the sample paths of the
Poisson processes
$\mathcal{N}_{A,c,n}$, $\mathcal{N}_{B,c,n}$, $\mathcal{N}_{T,c}$ above.
It can then be seen \revU{from the above} that, when $\tvecX^{(N)}(0) \in E^1$,
then the sample paths of the process $(\tvecX^{(N)}(t), \vecY^{(N)}(t))$
lie in $D_{E^1\times\Omega}$. 
%

We now rewrite these equations in a more compact form, working
componentwise.
When possible, the $N$ superscript has been omitted for conciseness.
\begin{align}
\Xck(t) &= \Xck(0) + \Ack{c,n-1}(t) - \Ack{c,n}(t) + \Dck{c,n+1}(t) - \Dck{c,n}(t), \label{eqn_Xk} \\
\Yc(Nt) &= \Yc(0) + N_c\sum_{n=1}^\infty\Dck{c,n}(t) - \Vc(t), \label{eqn_Y}   
\end{align}
where
\begin{align}
\Ack{c,n}(t) &= \frac{1}{N_c} \pNa\Bigl( N\lambda_c \int_0^t \Xck(s) ds \Bigr), \label{eqn_APoiss} \\
\Dck{c,n}(t) &= \frac{1}{N_c} \int_0^t \mathbbm{1}\{\vecY(Ns^-) \in \Ominc \} \dPoiss_{B,c,n}\Bigl( N \nu_c \int_0^s \Xck(u) du\Bigr),  \label{eqn_DPoiss} \\
\Vc(t) & = \int_0^t \Indic{\vecY(Ns-) \in \Opluc}\dPoiss_{T,c}\Bigl( N \mu_c \int_0^s du \Bigr), \label{eqn_TxPoiss}
\end{align}
for each $c \in C$, $n \in \ZO$.

Equation~\eqref{eqn_Xk} expresses the change in the fraction
of class-$c$ nodes with $n$ packets, i.e., arrivals to a queue
with $n-1$ packets or departures from a queue with $n+1$ packets
increment this fraction whereas arrivals or departures to component~$n$
result in a decrement.
Equation~\eqref{eqn_Y} describes the dynamics of the class activity
process and $\Yc$ is incremented whenever there is a back-off
from any non-empty class-$c$ queue, and returns to inactivity when the
corresponding packet has been transmitted.
Equation~\eqref{eqn_APoiss} defines the sequence of arrivals by fluid
time~$t$, the (stochastic) intensity is proportional to the fraction
of queues for each component, with the convention $A_{c,-1}(t) \equiv 0$
for all~$t$. 
Equation~\eqref{eqn_DPoiss} expresses the back-offs as an integral
of a previsible process over a Poisson process, again with intensity
varying according to the fraction of the corresponding component. 
Note that $D_{c,0}(t) = 0$ as no back-offs can occur from empty queues.
Finally the transmission process~\eqref{eqn_TxPoiss} is again
an integral of a previsible process over a scaled version of the
original Poisson process.

The expressions \eqref{eqn_Xk}--\eqref{eqn_TxPoiss} are preferably
written in the following martingale form:
\begin{align}
\vecX(t) = & \vecX(0) + \sum_{\cCC} \sum_{n=0}^\infty \vece^{c,n+1}_{c,n} \frac{\lambda_c}{p_{c,N}} \int_0^t \Xck(s) ds \nonumber \\
& + \sum_{\cCC} \sum_{n=1}^\infty \vece^{c,n-1}_{c,n} \frac{\nu_c}{p_{c,N}} \int_0^t 
\mathbbm{1}\{ \vecY(Ns^-) \in \Ominc \} \Xck(s) ds + \vecXI(t), \label{eqn_MartX} \\
\Yc(Nt) = & \Yc(0) + N\nu_c \int_0^t \mathbbm{1}\{\vecY(Ns^-) \in \Ominc\} (1 - X_{c,0}(s)) ds \nonumber \\
& - N\mu_c \int_0^t \Indic{\vecY(Ns^-) \in \Opluc} ds + \psi_c(t), \label{eqn_MartY}
\end{align}
where 
\begin{align} 
\Xi_{c,n}(t) = & \frac{1}{N_c} \ls M_{A,c,n-1}(t) - M_{A,c,n}(t) + M_{B,c,n+1}(t) - M_{B,c,n}(t) \rs, \label{eqn_Marteta} \\
\psi_c(t) = & \sum_{n=1}^\infty M_{B,c,n}(t) - M_{T,c}(t), \label{eqn_Martpsi} \\
M_{A,c,n}(t) = & \tNa \Bigl( N \lambda_c \int_0^t \Xck(s) ds \Bigr), \label{eqn_MartA} \\
M_{B,c,n}(t)  = & \int_0^t \Indic{ \vecY(Ns^-) \in \Ominc } \dtNb \Bigl(N\nu_c \int_0^s \Xck(u) du\Bigr), \label{eqn_MartB} \\
M_{T,c}(t) = & \int_0^t \Indic{ \vecY(Ns^-) \in \Opluc} \dtNT\Bigl(N \mu_c \int_0^s du \Bigr). \label{eqn_MartT}
\end{align}

\noindent
\revR{By means of arguments along the same lines as given
in~\revU{\cite{Kurtz80}}, it can be proved that the processes $M_{A,n}$,
$M_{B,n}$, $M_T$, and $\vecXI$ are locally square integrable martingales.}

\subsubsection{Representation of {\bf Y} }
\label{Sec_Gen_FL_YRand}

Since the weak limit of the sequence of processes $\vecY^{(N)}$ does not exist in~$D$,
we introduce the cumulative time process $\vecalpha^{(N)}$,
where $\alpha_\omega^{(N)}(t)$ denotes the cumulative time spent
in state~$\omega$ by the activity process $\vecY^{(N)}$ in the
interval \revT{$[0,Nt]$}. 
That is, for each $\omega \in \Omega$, $t \geq 0$,
\begin{equation}
\alpha_\omega^{(N)}(t) \doteq
\frac{1}{N} \int_0^{Nt} \mathbbm{1}\{\vecY^{(N)}(s^-) = \omega\} ds.
\label{eqn_alphadefn}
\end{equation}
Observe that $\sum_{\omega}\alpha_{\omega}^{(N)}(t) = t$,
$\alpha_\omega^{(N)}(0) = 0$, and $\alpha_\omega^{(N)}(t)$ is a continuous,
increasing and unit Lipschitz function for every $\omega \in \Omega$.
Indeed, $\vecalpha^{(N)}$ is unit Lipschitz under the $\sup$ norm metric.
\revR{Hence, the realizations of $\vecalpha^{(N)}$ lies in $\LipIncr[0,T]$,
where
\begin{align*}
\LipIncr[0,T]= 
\big\{l: [0,T] \rightarrow \mathbb{R}^{\abs{\Omega}}_+: &
\:l(0)=\veczero,\: \norm{l(t_1)-l(t_2)} \leq |t_1-t_2|, \\
&\: l_{\omega}(t_1)\geq l_{\omega}(t_2),\: \forall t_1 \geq t_2,\:
\omega\in\Omega\big\}.
\end{align*}}
By its definition, $\alpha_\omega^{(N)}(t)$ has left and right derivatives
everywhere taking values in $\{0,1\}$.
Note that $\alpha_\omega^{(N)}(t)$ increases at rate~$1$
if $\vecY^{(N)}(t)=\omega$, and~$0$ otherwise. 
Taking left derivatives along a sample path, it follows that
\[
\dot\alpha_\omega^{(N)}(t) =
\begin{cases}
\mathbbm{1}\{\vecY^{(N)}(Nt) = \omega\},
&\text{if }\vecY^{(N)}(Nt) \text{ is continuous at } Nt, \\
\mathbbm{1}\{\vecY^{(N)}(Nt^-) = \omega\}, &\text{otherwise}.
\end{cases}
\]
We may therefore rewrite~(\ref{eqn_MartX}) by substituting the left
derivatives of $\vecalpha^{(N)}$ in the intensity function (compensator),
\begin{align}
\vecX^{(N)}(t)  = &\vecX^{(N)}(0) + \sum_{\cCC} \sum_{n=0}^\infty \vece^{c,n+1}_{c,n} \frac{\lambda_c}{p_{c,N}} \int_0^t X^{(N)}_{c,n}(s) ds \nonumber \\
& + \sum_{\cCC} \sum_{n=1}^\infty \vece^{c,n-1}_{c,n} \frac{\nu_c}{p_{c,N}} \int_0^t
\sum_{\omega \in \Ominc} \dtalpo^{(N)}(s) X^{(N)}_{c,n}(s) ds +
\vecXI^{(N)}(t). \label{eqn_alphaMartX}
\end{align}

\revU{The next propostion ensures the existence of a weak limit for the
sequence of joint processes $\big(\vecX^{(N)}(t),\vecalpha^{(N)}(t)\big)$.}

\begin{proposition}
\label{Prop_XalpRerlCompact}
The sequence of joint \revU{processes}
$\big(\vecX^{(N)}(t),\vecalpha^{(N)}(t)\big) \in D_{E^1}[0,\infty) \times
L_{\uparrow}[0,\infty)$ is relatively compact.
\end{proposition}

Observe that, \revB{as a consequence of \cite[Exercise~6, page~41]{Billingsley68}, 
in order to prove Proposition~\ref{Prop_XalpRerlCompact} it suffices
to show the tightness for the marginals}, i.e.,
the tightness of $\vecalpha^{(N)}(t)$ and of $\vecX^{(N)}(t)$. 
\revU{The proof of the tightness of $\vecalpha^{(N)}(t)$ is immediate,
and the tightness of $\vecX^{(N)}(t)$ is a consequence of the relative
compactness of~$E$ established in Lemma~\ref{ECompact} and the
observation that the jumps happen at rate $O(N)$ and have size $O(1/N)$ 
\cite{Bramson98,Stolyar05}.}

\subsubsection{Mean-field limit characterization}
\label{Sec_Gen_FL_FL}

\revB{We just proved that the sequence of processes
$\big(\vecX^{(N)}(t),\vecalpha^{(N)}(t)\big)$ is relatively compact,
and we now derive the unique characterization
$\lb\vecx(t), \vecalpha(t)\rb$ of the limiting process of any
converging subsequence}. 
This suffices to derive a weak limit for the sequence
$\big(\vecX^{(N)}(t),\vecalpha^{(N)}(t)\big)$.
As a further step towards proving Theorem~\ref{Theo_MFGen}, we first
establish an intermediate mean-field limit, which somewhat resembles
that in \cite[Lemma 1]{HK94}).
Specifically we describe the steps needed to obtain the following result.
\begin{proposition}
\label{prop_MFA}
\revB{Consider any convergent subsequence of
\[(\vecX^{(N)}(t),\vecalpha^{(N)}(t)) \in D_E[0,\infty) \times \LipIncr[0,\infty),\] 
its limit $\big(\vecx(t),\vecalpha(t)\big)$ satisfies the differential equation}
\begin{align*}
\vecx(t) = & \vecx(0) + \sumc \frac{\lambda_c}{p_c} \sum_{n=0}^{\infty} \vece^{c,n+1}_{c,n} \int_0^t x_{c,n}(s)ds \\
& + \sumc\frac{\nu_c}{p_c} \sum_{n=1}^{\infty} \vece^{c,n-1}_{c,n} \int_0^t x_{c,n}(s)\sum_{\omega \in \Omega_{-c}}\dot\alpha_{\omega}(s)ds. 
\end{align*}
\end{proposition}

Observe that Equation~\eqref{eqn_alphaMartX} determines $\vecX^{(N)}(t)$
as in the following sum, of an initial term plus sample paths lying
in $D_E[0,t]$:
\begin{equation}
\vecX^{(N)}(t) =
\vecX^{(N)}(0) + \vecI_A^{(N)}(t) + \vecI_B^{(N)}(t) + \vecXI^{(N)}(t),
\label{eqn_MartCts}
\end{equation}
where
\begin{eqnarray*}
\vecI_A^{(N)}(t) & \doteq & \sumc\frac{\lambda_c}{p_{c,N}} \sum_{n=0}^{\infty} \vece^{c,n+1}_{c,n} \int_0^t X_{c,n}^{(N)}(s)ds \\
\vecI_B^{(N)}(t) & \doteq & \sum_{\cCC} \frac{\nu_c}{p_{c,N}} \sum_{n=1}^\infty \vece^{c,n-1}_{c,n} \int_0^t \sum_{\omega \in \Ominc} \dtalpo^{(N)}(s)  X^{(N)}_{c,n}(s) ds.
\end{eqnarray*}
Furthermore if $\vecx(t)$, $\vecalpha(t)$ are weak limits
of the processes $\vecX^{(N)}(t)$, $\vecalpha^{(N)}(t)$, then we make corresponding
definitions for $\vecI_A(t)$, $\vecI_B(t)$.
Further define $\vecI^{(N)}(t) \doteq
\vecX^{(N)}(0) + \vecI_A^{(N)}(t) + \vecI_B^{(N)}(t)$ and let the process $\vecI$
be the corresponding limit.

In order to prove Proposition~\ref{prop_MFA}, we need to go through
the following steps:
\begin{enumerate}
\item[(a)] Show that the weak limits of $\vecX^{(N)}(t)$
and $\vecI^{(N)}(t)$ coincide. 
\item[(b)] Derive the weak limit of $\vecI^{(N)}(t)$.
\end{enumerate}

For step~(a), we apply the Continuous Mapping Theorem,
see \cite[Theorem 5.1, page 30]{Billingsley68} to the weak limits
for the terms on the right hand side of~\eqref{eqn_MartCts} as follows.
First recall \cite[Theorem 3.1, page 27]{Billingsley99} which states 
that if $\big(\vecI^{(N)},\vecX^{(N)}\big)$ are random elements
of $D_E[0,\infty) \times D_E[0,\infty)$  (defined on some common
probability space) and taking values in \revB{$(E,\rho)$}, then if 
$\vecI^{(N)} \Rightarrow \vecI$ and it holds that
$d(\vecI^{(N)},\vecX^{(N)}) \Rightarrow 0$, it follows that
$\vecX^{(N)} \Rightarrow \vecI$.
To apply the above result to~\eqref{eqn_MartCts}, we will take $d$
to be the Skorohod metric on $D_E[0,\infty)$ as defined
in \cite[Section 3.5]{EK86}.
In order to show that $\vecX^{(N)}$ and $\vecI^{(N)}$ have the same
weak limit, it will be enough to establish the following lemma
which is proved in Appendix~\ref{App_prop_dZ}.

\begin{lemma}
\label{prop_dZ}
Given  any $\eta > 0$ and $T > 0$,
\[
\lim_{N\rightarrow\infty} \mathbb{P}\Big\{\sup_{t \in [0,T]}
\rho(\vecXI^{(N)}(t), \boldsymbol{0}) > \eta \Big\} = 0.
\]
\end{lemma}

\noindent
It can be shown that Lemma~\ref{prop_dZ} implies that
$d(\vecX^{(N)},\vecI^{(N)}) \Rightarrow 0$ \revT{by leveraging the
arguments in \cite[Section 5, p.117]{EK86}.}

We now continue with step~(b).
Observe that $\vecI^{(N)}(t)$ is a sequence of continuous paths.
We will establish convergence of this latter sequence in $C_E[0,\infty)$
under the \revU{local} uniform metric using the Continuous Mapping Theorem.
This establishes convergence in~$D$ (the Skorohod topology relativized
to~$C$ coincides with the \revU{local} uniform topology).

Given a subsequence of processes
$(\vecX^{(N_m)},\vecalpha^{(N_m)}) \Rightarrow (\vecx,\vecalpha)$
(without loss of generality we will take this subsequence to be the
whole sequence), we obtain the corresponding weak limit of $\vecI^{(N)}$
provided that the given mapping is continuous. 
We now prove the continuity of such mapping. 
\revR{Consider a sequence
\[
\lb{\hat{\vecX}}^{(N)}(t),\hat{\vecalpha}^{(N)}(t)\rb \rightarrow
\lb\hat{\vecx}(t), \hat{\vecalpha}(t)\rb
\]
in the product topology. Then, for all $t > 0$, we have that
\begin{equation}
\int_0^t \hat{X}^{(N)}_{c,n}(s)ds \rightarrow \int_0^t \hat x_{c,n}(s) ds,
\label{eqn_xIntegral}
\end{equation}
and
\begin{equation}
\sum_{\omega \in \Ominc} \int_0^t \hat X^{(N)}_{c,n}(s) \dot{\hat{\alpha}}^{(N)}_\omega(s) ds \rightarrow
\sum_{\omega \in \Ominc} \int_0^t \hat x_{c,n}(s) \dot{\hat{\alpha}}_\omega(s) ds.
\label{eqn_alphaxIntegral}
\end{equation}
The relations above hold since the integral operator is continuous
and $\dot{\hat{\alpha}}_\omega(t) \leq 1$ for every $t \geq 0$.}

Note that $\hat X^{(N)}_{c,n}(s) \in D_{[0,1]}[0,\infty)$ is Lebesgue
integrable over any finite interval and so the first integral is
well-defined.
Since the processes $\hat \alpha^{(N)}_\omega$ are unit Lipschitz and increasing,
the integrals in~\eqref{eqn_alphaxIntegral} also exist.
\revR{Equations~\eqref{eqn_xIntegral} and~\eqref{eqn_alphaxIntegral}
only establish pointwise convergence. 
However, given a sequence $\{u^{(N)}(t)\}_N$ of nondecreasing 
Lipschitz continuous functions with uniformly bounded Lipschitz
constant that converges pointwise to a limiting Lipschitz function $u(t))$,
it can be shown that the convergence holds uniformly as well
due to the triangle inequality and the common Lipschitz constant, i.e.,
for every $\epsilon > 0$ there exists $N_\epsilon$ such that
$\sup_{t \in [0,T]} \abs{u^{(N)}(t) - u(t)} < \epsilon$ for every
$N \geq N_\epsilon$.}

\revB{Since the various components are bounded, due to the Weierstrass
M-test \cite[Theorem 7.10, p.~148]{Rudin76}, the sum of the components
over~$n$ and~$c$ converges to the sum of the limiting components.}
Hence, it holds that
\begin{align*}
\sup_{t \in [0,T]} \rho( I^{(N)}_{A}(t), I_{A}(t) ) \rightarrow 0, \qquad
\sup_{t \in [0,T]} \rho( I^{(N)}_{B}(t), I_{B}(t) ) \rightarrow 0. 
\end{align*}
\revB{We just proved that the mapping is continuous and thus,
due to the Continuous Mapping Theorem, we deduce that
$ \vecI^{(N)} \Rightarrow \vecI $ as required.}
\revB{Since $T$ is arbitrary this implies convergence in $C_E[0,\infty)$
due to \cite[Lemma 3, p.173]{Billingsley99}, and hence in $D_E[0,\infty)$.}


\subsubsection{Weak limit characterization of $\vecalpha$}
\label{Sec_Gen_FL_WeakLimalpha}

\noindent 
Proposition~\ref{prop_MFA} determines any weak limit~$\vecx$ as the
solution to a differential equation for which the corresponding limit
occupancy measure~$\vecalpha$ is given.
We now proceed to characterize this latter limit in terms of the
stationary measure of the activity process.

Recall the properties satisfied by the processes $\alpha_\omega^{(N)}$.
In particular, for every $t \geq 0$,
\begin{equation}
\label{RelationCumOccAct}
\alpha_\omega^{(N)}(t)\geq 0, \qquad
\sum_{\omega \in \Omega} \alpha_\omega^{(N)} = t, \qquad
\sum_{\omega \in \Omega} \dot\alpha_\omega^{(N)}(t) = 1. 
 \end{equation}
We begin by substituting $\balp^{(N)}$ into~\eqref{eqn_MartY}
and obtain that $\psi_c(t)$ equals
\begin{align*}
Y_c^{(N)}(Nt) - Y_c^{(N)}(0) +  N \mu_c \sum_{\omega \in \Opluc} \int_0^t \dtalp_\omega^{(N)}(s) ds \\
- N \nu_c \sum_{\omega \in \Ominc}\int_0^t \dtalp_\omega^{(N)}(s) \lb 1 - X^{(N)}_{c,0}(s) \rb ds,    
\end{align*}
which therefore is a martingale.
Dividing by~$N$ and taking weak limits,
we obtain that 
\begin{align}
&\int_0^t \nu_c (1-X_{c,0}^{(N)}(s)) \sum_{\omega \in \Omega_{-c}}
\dot\alpha_\omega^{(N)}(s)ds - \int_0^t \mu_c \sum_{\omega \in \Omega_{+c}}
\dot\alpha_\omega^{(N)}(s)ds \Rightarrow 0, \label{YmartingalePrelimit}
\end{align}
for every $t \geq 0$,
since $\psi_c/N \Rightarrow 0$ due to Doob's inequality.

We apply the Continuous Mapping Theorem to~\eqref{YmartingalePrelimit}
as explained in the discussion before Lemma~\ref{prop_dZ}, 
and use the limits established in Equations~\eqref{eqn_xIntegral}
and~\eqref{eqn_alphaxIntegral}, to obtain that for every $t \geq 0$
\begin{equation}
0 = \int_0^t \nu_c (1-x_{c,0}(s))
\sum_{\omega \in \Omega_{-c}} \dtalp_\omega(s) ds -
\int_0^t \mu_c \sum_{\omega \in \Omega_{+c}} \dtalp_\omega(s) ds.
\label{FLCumulativeTransActivProc}
\end{equation}

This leads to the following corollary proved
in Appendix~\ref{proofalphacorollary}.

\begin{corollary}
\label{alphacorollary}
The function $\alpha_{\omega}(t) \in C[0,\infty)$ is differentiable
almost everywhere and
\begin{equation}
\dtalp_{\omega}(t) = \pi_{\vecx_0(t)}(\omega).
\label{eqn_alphaCharacterization2}
\end{equation}
\end{corollary}

Having characterized the weak limit of the process $\balp$, we obtain
Theorem~\ref{Theo_MFGen} by substituting~\eqref{eqn_alphaCharacterization2}
in Proposition~\ref{prop_MFA}.
  
\subsection{Analysis of the mean-field limit}
\label{Sec_Gen_ExisUniq}

In the previous subsection we proved Theorem~\ref{Theo_MFGen},
establishing the convergence of $\vecX^{(N)}(t)$
to $\vecx(t) \in C_{E}[0,\infty)$, the solution of the initial-value
problem~\eqref{eqn_MFDiffEqnGen}.
In this subsection \revR{we show that $\vecx(t)$ takes values in $E^1$}
and we present the proof of Theorem~\ref{Theo_EquilibChar}.
Specifically, we show that a solution to~\eqref{eqn_MFDiffEqnGen}
exists and is unique.
Moreover, we provide a general condition yielding the existence
of a unique \revT{fixed} point~$\vecx^*$. \\ 
 
\noindent
{\it Transient behavior}. 
\revR{In Appendix~\ref{App_LipContH} we prove that the function
$H(t,\vecx(t))$ defined in~\eqref{eqn_MFDiffEqnGen} is Lipschitz
continuous in $\vecx(t)$, and consequently in~$t$ as well
since it depends on time only through~$\vecx(t)$. 
Thanks to the results in~\cite{DP86}, we have that even if $E$ is
infinite-dimensional, the Lipschitz continuity yields that a solution
$\vecx(t) \in C_E[0,\infty)$ of~\eqref{eqn_MFDiffEqnGen} exists
and is unique. 
Given a solution $\vecx(t)$, it holds that
\[
\frac{\partial}{\partial t} \sum_{n = 0}^{\infty} x_{c,n}(t) =
\sum_{n = 0}^{\infty} \frac{\partial}{\partial t}  x_{c,n}(t) =
\sum_{n = 0}^{\infty} H_{c,n}(\vecx(t)) = 0,
\]
where summation and derivatives can be interchanged since
$\sum_{n = 0}^{\infty} \frac{\partial}{\partial t}  x_{c,n}(t)$
converges uniformly \cite[Theorem 7.17]{Rudin76}.
Hence, the set $E^1$ is positive invariant for the initial-value
problem~\eqref{eqn_MFDiffEqnGen} and given an arbitrary $\vecx(0)\in E^1$,
the solution $\vecx(t)$ remains in~$E^1$ for every $t \geq 0$.} \\

\noindent
{\it Limiting behavior}.
In order to prove Theorem~\ref{Theo_EquilibChar}, we observe that when
a \revT{fixed} point~$\vecx^*$ exists, it must satisfy the relation
$H(\vecx^*) = \veczero$ component-wise, i.e., for every $\cCC$
\[
\lambda_c (x_{c,n-1}^* - x_{c,n}^*) +
\nu_c \pi_{\vecx_0^*}(\Omega_{-c}) (x_{c,n+1}^* - x_{c,n}^*) = 0,
\]
for all $n > 0$, and
\[
- \lambda_c x_{c,0}^* + \nu_c \pi_{\vecx_0^*}(\Omega_{-c}) x_{c,1}^* = 0.
\]
Thus
\[
\lambda_c x_{c,n}^* = \nu_c \pi_{\vecx_0^*}(\Omega_{-c}) x_{c,n+1}^*,
\qquad \forall\: n \in \NatZero,
\]
yielding
\begin{equation}
\label{eqn_FormulaXi}
x_{c,n}^* = \xi_c^n x_{c,0}^*, \qquad 
\xi_c = \frac{\lambda_c}{\nu_c \pi_{\vecx_0^*}(\Omega_{-c})}
\end{equation}
for every $n \in \NatZero$.
In other words, any \revT{fixed} point $\vecx^*$ must satisfy the
above geometric relation.
In addition, such a \revT{fixed} point $\vecx^*$ lies in~$E^1$
if and only if $\xi_c = 1 - x_{c,0}^* < 1$ for every $\cCC$, i.e.,
$\vecx_0^* > \veczero$.
Hence, thanks to~\eqref{eqn_FormulaXi}, the uniqueness of~$\vecxi$
yields the same for~$\vecx^*$. 

According to the product-form solution in~\eqref{Eqn_piDef} and the
definition of throughput~\eqref{Eqn_throuDef}, it holds that
\[
\theta_c(\vecxi) = \pi_{\vecx_0^*}(\Omega_{+c}) =
\sum_{\omega \in \Omega_{+c}} \pi(\omega;\vecxi).
\]
In order for $\vecx^*$ to be a \revT{fixed} point, $\vecxi$ must solve
$\theta_c(\vecxi) =  \rho_c$ for every $\cCC$, i.e., the fraction of time
that class~$c$ is active in stationarity must be equal to its load.
Noting that $\vecrho \in \mbox{int}(\Gamma)$, the global invertibility
of the throughput map~\eqref{Eqn_InverseDef} implies that $\vecxi$ has
to satisfy $\vecxi = \veceta(\vecrho)$.
Observe that the existence and uniqueness of~$\vecxi$ follows from the
global invertibility property of the map $\vectheta(\cdot)$.
We conclude that $\vecx^*$ as in~\eqref{EquilibriumRelation} is the
unique \revT{fixed} point of the initial-value
problem~\eqref{eqn_MFDiffEqnGen}, and that $\vecx^* \in E^1$
if and only if $\veceta(\vecrho) < \vece$.
(As a side-remark, we mention that the initial-value problem would not
have any \revT{fixed} point, if the assumption
$\vecrho \in \mbox{int}(\Gamma)$ were not satisfied.
This makes sense since the latter condition is necessary for the queue
lengths to be stable.)

As an example, let us revisit the square network displayed
in Figure~\ref{fig_SquareGraph} and assume that 
\[
\revU{\max\{\rho_1,\rho_3\} + \max\{\rho_2,\rho_4\} < 1},
\]
i.e., $\vecrho \in \mbox{int}(\Gamma)$.
In order to identify the unique \revT{fixed} point of the initial-value 
problem~\eqref{eqn_MFDiffEqnGen}, the following system of equations has
to be solved:
\begin{equation}
\label{eqn_SquareNetworkXiFormula}
\xi_c = \frac{\lambda_c}{\nu_c \pi_{\vecx_0^*}(\Omega_{-c})}, \qquad c=1,2,3,4, 
\end{equation}
where
\begin{align*}
&\pi_{\vecx_0^*}(\Omega_{-1}) = \frac{1}{Z(\vecxi)}\big(1 + \sigma_3 \xi_3\big), \qquad\pi_{\vecx_0^*}(\Omega_{-2}) = \frac{1}{Z(\vecxi)}\big(1 + \sigma_4 \xi_4\big), \\
&\pi_{\vecx_0^*}(\Omega_{-3}) = \frac{1}{Z(\vecxi)}\big(1 + \sigma_4 \xi_4\big), \qquad\pi_{\vecx_0^*}(\Omega_{-4}) = \frac{1}{Z(\vecxi)}\big(1 + \sigma_3 \xi_3\big),
\end{align*}
and
\[
Z(\vecxi) = 1 + \sum_{c=1}^{4} \sigma_c \xi_c +
\sigma_1\sigma_3\xi_1\xi_3 + \sigma_2\sigma_4\xi_2\xi_4.
\]
The \revT{fixed} point exists if and only if $\vecxi < \vece$. \\

In this subsection, we proved the convergence of the population process
towards the solution of a deterministic initial-value problem,
and established necessary and sufficient conditions for the existence
and uniqueness of a \revT{fixed} point.
In the following subsection we aim to exploit the \revT{fixed} point
so as to obtain an asymptotically exact approximation for the
stationary performance measures of the system, with a focus on the
case of a complete interference graph.

\section{The complete interference graph case}
\label{Sec_Compl}

\revU{In this section we focus on the analysis of key performance
measures such as the stationary queue length, waiting-time, 
and sojourn-time distributions.
In particular, we will exploit the \revT{fixed} point~$\vecx^*$ 
derived in Theorem~\ref{Theo_EquilibChar} to obtain approximations
that are asymptotically exact in the mean-field regime for the case
of a complete interference graph.
Simulation experiments that will be presented in Section~\ref{Sec_Num} suggest that the asymptotic
results are valid for general interference graphs as well.}

Specifically, we will prove Theorem~\ref{Theo_Comp_IL}, i.e., we will
show that for any initial state $\vecx^\infty \in E^1$ with finite 
mass the solution $\vecx(t)$ of the initial value problem satisfies
$\vecx(t) \to \vecx^*$ as $t \to \infty$, and that limits interchange
so as to obtain
\begin{equation}
\label{eqn_Com_Objective}
\lim_{N\rightarrow\infty} \lim_{t\rightarrow\infty}
\mathbb{P}\{\rho(\vecX^{(N)}(t),\vecx^*)>\epsilon\} = 0, \quad
\forall \: \epsilon>0. 
\end{equation}
Due to the class symmetry, relation~\eqref{eqn_Com_Objective} implies that
\[
\lim_{N\rightarrow\infty} \lim_{t\rightarrow\infty}
\mathbb{P}\{Q^{(N)}_{c,k}(t) = n\} = x^*_{c,n}, \quad
\forall \: \cCC, k = 1,\ldots,N_c.
\]
In order to establish~\eqref{eqn_Com_Objective}, we need several steps.
\begin{itemize}
\item In Section~\ref{Sec_Compl_GS} \revT{we will show that for any
initial state $\vecx^\infty \in E^1$ (subject to a finite-mass
condition that will be described later) the solution $\vecx(t)$ of the
IVP~\eqref{eqn_MFDiffEqnGen_com} satisfies $\vecx(t) \to \vecx^*$
as $t \to \infty$, which implies}
\begin{equation}
\lim_{t\rightarrow\infty} \lim_{N\rightarrow\infty}
\mathbb{P}\{\rho(\vecX^{(N)}(t),\vecx^*)>\epsilon\} = 0, \quad
\forall \: \epsilon>0.
\label{eqn_GlobalStability}
\end{equation}
\item Equation~\eqref{eqn_Com_Objective} \revU{implicitly assumes} that
the population process in the finite system has a stationary distribution,
hence conditions for the positive recurrence of the prelimit processes
are established in Section~\ref{Sec_Compl_PR}.
\item \revT{In Section \ref{Sec_Compl_TightIL} we finally prove
Theorem~\ref{Theo_Comp_IL} and specifically~\eqref{eqn_Com_Objective},
i.e., we show that the limits in~\eqref{eqn_GlobalStability} may be
interchanged as sketched in Figure~\ref{FigSquare}.}
\end{itemize}

\revU{Throughout this section} we assume the interference graph~$G$
to be complete and condition~\eqref{eqn_Gen_Assum} to hold.
For a complete interference graph, it holds that
\[
\theta_c(\vecxi) = \pi(\vece_c;\vecxi) =
\frac{\xi_c\sigma_c}{1 + \sum_{d \in \CC}\xi_d\sigma_d}.
\]
Thus the \revT{fixed} point~$\vecx^*$ has the following closed-form
expression, and \eqref{eqn_Gen_Assum} is indeed necessary
and sufficient so as to guarantee $\vecxi < \vece$.

\begin{corollary}
\label{Theo_EquilibChar_comp}
The unique \revT{fixed} point in~$E^1$ of the 
initial-value problem~\eqref{eqn_MFDiffEqnGen} is given
by $\vecx^* \in E^1$, where
\begin{equation}
x_{c,n}^* = (1 - \xi_c) \xi_c^n, \qquad 
\xi_c =
\frac{\lambda_c}{\nu_c\big(1 - \sum_{d\in\CC}\frac{\lambda_d}{\mu_d}\big)}.
\label{eqn_EquilibriumRelation_comp}
\end{equation}
\end{corollary}
 
\subsection{Global stability of $\vecx^*$}
\label{Sec_Compl_GS}

We begin by observing that, when the interference graph is complete,
Theorem~\ref{Theo_MFGen} specializes as follows.

\begin{corollary}
\label{Thm_MFGen_comp}
Assume $G = (\mathcal{C},\mathcal{E})$ to be complete,
$\vecX^{(N)}(0) \xrightarrow{N\rightarrow\infty} \vecx^{\infty} \in E^1$
and $\lim_{N\rightarrow\infty} p_{c,N} = p_c >0$ for every
$c\in\mathcal{C}$.
Then the sequence of processes
$(\vecX^{(N)}(t))_{N \geq 1} \in D_{E^1}[0,\infty)$ has a continuous
limit $\vecx(t) \in C_{E^1}[0,\infty)$ which is determined by the
unique solution of the initial-value problem
\begin{equation}
\frac{d \vecx(t)}{dt} = H(\vecx(t)), \qquad \vecx(0) = \vecx^{\infty}, 
\label{eqn_MFDiffEqnGen_com}
\end{equation}
where the function $H(\cdot)$ is defined by
\[
H(\vecx) = \sumc \frac{1}{p_c}
\Big(\lambda_c \sum_{n=0}^{\infty} x_{c,n}\vece^{c,n+1}_{c,n} +
\pi_{\vecx_0}^b \nu_c \sum_{n=1}^{\infty} x_{c,n}\vece^{c,n-1}_{c,n}\Big),
\]
and $\pi_{\vecx_0}^b =
\frac{1}{1 + \sum_{c\in \CC}\frac{\nu_c}{\mu_c}(1 - x_{c,0})}$.
\end{corollary}

A critical role in establishing~\ref{eqn_GlobalStability} is played
by the notion of the {\it mass} of a population vector, defined as
the average number of packets in the buffer of the various nodes.
It can be shown that for any initial state $\vecx^\infty \in E^1$
with finite mass the solution $\vecx(t)$ of the initial-value problem
satisfies $\vecx(t) \to \vecx^*$ as $t \to \infty$,
which implies~\eqref{eqn_GlobalStability}, i.e., the fixed point
$\vecx^* \in E^1$ as derived in Corollary~\ref{Theo_EquilibChar_comp}
is a stable equilibrium point of the initial-value
problem~\eqref{eqn_MFDiffEqnGen_com}.

\revT{
A classic result (see for instance~\cite[p.~1846 {\it seq}]{Whitt85},
but probably dating back to a few decades before that), states that
the above property is sufficient to establish that
$\delta_{\vecx^*} \in \mathbb{P}(E^1)$, the probability measure
concentrated in~$\vecx^*$, is the unique {\it invariant distribution} 
for~\eqref{eqn_MFDiffEqnGen_com}.
A distribution $\vecd$ is invariant for an initial-value problem if, 
given that the initial condition is distributed according to~$\vecd$,
the solution at time~$t$ is distributed according to~$\vecd$ as well
for every $t \geq 0$.}

\revU{We only briefly outline the steps necessary in order to
prove~\eqref{eqn_GlobalStability}, and refer to~\cite{FabioThesis18}
for detailed proofs.
\begin{itemize}
\item The first step is to show that the mass of the \revT{fixed}
point~$\vecx^*$ is finite and the mass of a solution $\vecx(t)$
of~\eqref{eqn_MFDiffEqnGen_com} is a Lipschitz continuous function in~$t$. 
\item The second step is to establish an invariant property for the
initial-value problem~\eqref{eqn_MFDiffEqnGen_com}.
This property allows us to couple different solutions
of~\eqref{eqn_MFDiffEqnGen_com}, and plays a key role in the proof
of~\eqref{eqn_GlobalStability}.
\item The final step then leverages the above two steps so as to show
that, if $\vecx(0)$ has finite mass, $\vecx(t) \rightarrow \vecx^*$.
The \revT{fixed} point $\vecx^*$ is thus a {\it globally stable}
equilibrium point. 
\end{itemize}

It is worth observing that the invariance property in the second step
has only been established for the complete interference graph,
and does not extend to arbitrary interference graphs.
Note however that the invariance property is only a convenient proof
technique used in arguing global stability, and not a necessary
condition for global stability to hold.
}

\subsection{Positive recurrence}
\label{Sec_Compl_PR}

As mentioned earlier, the final objective of this section is to obtain
an approximation for the stationary distribution of the queue length
processes at the various nodes in large systems. 
\revT{In particular, we will show that, when $N$ is large,
$\vecx^*\in E^1$ provides an approximation for the random variable
$\vecX_*^{(N)}$ whose law is given by the stationary distribution
$\pi_{\vecX_*^{(N)}}$ of the population process system with $N$~nodes.
First, however, it must be ensured that the population process has
a stationary distribution, i.e., that the random variable
$\vecX_*^{(N)}$ exists.
In this section, we will establish the positive recurrence of the
queue length process $\vecQ^{(N)}(t)$ under condition~\eqref{eqn_Gen_Assum},
for any $N \geq 1$.
Then, the positive recurrence of $\vecX^{(N)}(t)$ immediately follows
by construction.}

The complete interference graph assumption \revU{continues to play}
an important role in this analysis.
Indeed, since at every moment in time at most one node can be active,
the CSMA model can be described as a single-server polling system
with random routing.
Specifically, the switchover periods of the server correspond to the 
idle-state of the activity process, i.e., the intervals in which each
node is allowed to back-off. 
When a node completes its back-off period, the server visits that 
node and serves a packet, when present.
In particular, consider the polling system with the following features: 
\begin{itemize}
\item Single server and $N$~nodes.
The nodes are partitioned in $N_1^{(N)} + \ldots + N_C^{(N)}$ nodes
 with class-dependent features;
\item Packets arrive at the $k$-th node belonging to class~$c$ as
a Poisson process of rate $\lambda_{c}^{(N)}$.
\item The service policy is $1$-limited.
In particular, a single packet is served if present, otherwise 
the server leaves the node immediately;
\item When a class-$c$ node is visited and its buffer is nonempty,
a packet is served for an exponentially distributed time
with parameter~$\mu_c^{(N)}$;
\item The switchover time of the server between two consecutive visits
is exponentially distributed with parameter
\[
\nu \:\dot=\: \sumc N_c^{(N)} \nu_c^{(N)} = \sumc \nu_c;
\]
\item At the end of a switchover period, the probability for the
server to pick the $k$-th node belonging to class~$c$ is equal
to $p_{c,k} = \nu_c^{(N)}/\nu$.
\end{itemize}
Observe that the above-described polling system behaves exactly as the
CSMA model with complete interference and that the latter property is
crucial in establishing this connection.

Polling systems have been extensively studied in the literature.
In particular, the condition for the positive recurrence of the
above-described model has already been investigated. 
In \cite{DM95,Down98,FJ98}, the authors provide positive recurrence
conditions for general classes of polling systems via the fluid limit
analysis of the queue length processes \cite{DM95,Down98} and via
a direct approach~\cite{FJ98}. 
In the following proposition we recall a special case of the result
in~\cite[Theorem 2.3]{Down98} with the notation adapted to the CSMA
setting.

\begin{proposition}[Theorem 2.3 \cite{Down98}]
\label{Prop_Compl_PR}
Define
\[
\varrho \:\dot=\: \sum_{d \in \CC} \frac{\lambda_d}{\mu_d} +
\max_{\cCC} \frac{\lambda_c}{\nu_c}.
\]
\begin{enumerate}
\item[(i)] If $\varrho < 1$, the queue length process $\vecQ^{(N)}$ is
positive recurrent for every~$N$.
\item[(ii)] If $\varrho > 1$, the queue length process $\vecQ^{(N)}$
is transient for every~$N$.
\end{enumerate}
\end{proposition}

The proof in \cite[Theorem 2.3]{Down98} consists of analyzing the
fluid limit of a polling system, and we can conclude the positive
recurrence of the CSMA model due to the above-described equivalence.
Observe that the condition $\varrho < 1$ is not dependent on~$N$,
and is equivalent to condition~\eqref{eqn_Gen_Assum}, i.e.,
\[
\varrho < 1 \quad \iff \quad \xi_c < 1, \: \text{ for every } \cCC.
\]


\revU{Because of the equivalence, the pseudo-conservation law proved
in~\cite[Equation (5.34)]{BW89} for polling systems also applies
to the CSMA model with a complete interference graph, yielding
\[
\sumc \rho_c \big(1 - \frac{\lambda_{c}}{\nu_{c}(1 - \rho)}\big)
\mathbb{E}[W_c^{(N)}] =
\frac{1}{1-\rho} \Big(\rho \sumc \frac{\rho_c}{\mu_c} +
\sumc \frac{N_c \rho_c}{\nu_c}\Big), \label{PseudoCons1}
\]
where $\rho = \sum_{c=1}^N \rho_c$ and the random variable $W_c^{(N)}$
represents the time a packet spends waiting in the buffer  
of a class-$c$ node in stationarity in a system with $N$~nodes.
The above pseudo-conservation law is valid as long as the stationary
expectations of the waiting times are finite, which is ensured
by~\cite[Theorem 5.5]{DM95} as long as the stability
condition~\eqref{eqn_Gen_Assum} is satisfied.

It can be easily deduced from 
the above pseudo-conservation law that
\begin{align}
\limsup_{N\rightarrow\infty} \sumc \frac{\nu_c}{N_c} \mathbb{E}[W_c^{(N)}] = K_0 < \infty, \label{PseudoCons1_Upper2}
\end{align}
when \eqref{eqn_Gen_Assum} is satisfied.
The above bound will play a key role in proving tightness
and justifying an interchange of limits in the following subsection.
}
 
\subsection{Tightness and interchange of limits}
\label{Sec_Compl_TightIL}

In Proposition~\ref{Prop_Compl_PR} we showed that the queue length
process $\vecQ^{(N)}(t)$ has a stationary distribution
if condition~\eqref{eqn_Gen_Assum} is satisfied.
Denote it by the probability measure $\pi_{\vecQ_*^{(N)}}$, i.e., 
\[
\lim_{t\rightarrow\infty} \mathbb{P}\{\vecQ^{(N)}(t) = \vecq\} =
\pi_{\vecQ_*^{(N)}}(\vecq), \qquad \vecq \in \mathbb{N}_0^N.
\]
\revT{For finite~$N$ the process $\vecQ^{(N)}(t)$ \revR{uniquely
determines via a continuous map $\tvecX^{(N)}(t)$, hence the
population process $\tvecX^{(N)}(t)$} also has a stationary distribution.
In particular, define the function
$\vecg^{(N)}: \NatZero^N \rightarrow E^1$, where
\[
\big(\vecg^{(N)}(\vecq)\big)_{c,m} =
\frac{1}{N_c} \sum_{n=1}^{N_c} \mathbbm{1}\{q_{c,n}=m\}.
\]
Thus, we derive $\pi_{\vecX_*^{(N)}}$, the stationary distribution
of the population process.
Given $\vecx \in E^1$, we have that
\begin{equation}
\pi_{\vecX_*^{(N)}}\big(\vecx\big) =
\sum_{\vecq: \vecg^{(N)}(\vecq) = \vecx} \pi_{\vecQ_*^{(N)}}\big(\vecq\big).
\label{T04_eqn_RelQX}
\end{equation}
Denote by $\vecX_*^{(N)}$ the random variable distributed according
to~$\pi_{\vecX_*^{(N)}}$.
For every fixed $N \geq 1$, it holds that
$\vecX^{(N)}(t) \Rightarrow \vecX_*^{(N)}$ as $t \rightarrow \infty$.}

\revT{At this point we are close to proving
equation~\eqref{eqn_Com_Objective}, for which we need to show that 
$\vecX_*^{(N)} \Rightarrow \vecx^*$ as $N \rightarrow \infty$.
Indeed, in Subsection~\ref{Sec_Compl_GS} we characterized
$\vecx^* \in E^1$, and thanks to Proposition~\ref{Prop_Compl_PR},
we proved the existence of the sequence of random variables
$\{\vecX^{(N)}\}_{N \geq 1}$.
Two further steps are needed in order to show~\eqref{eqn_Com_Objective}. 
First, we need to actually show that the sequence of probability
measures $\{\pi_{\vecX_*^{(N)}}\}_{N \geq 1} \subseteq \mathbb{P}(E^1)$
has at least one limiting point. 
Then we need to show that each possible limiting point necessarily
coincides with $\delta_{\vecx^*} \in \mathbb{P}(E^1)$, the Dirac
measure concentrated in~$\vecx^*$. 
Once we settle these issues, we can finally conclude that
equation~\eqref{eqn_Com_Objective} holds, and thus the proof
of Theorem~\ref{Theo_Comp_IL} is complete.

Key to the first step described above is the following proposition
whose proof is presented in Appendix~\ref{App_Prop_Comp_Tightness}.}

\begin{proposition}
\label{Prop_Comp_Tightness}
The sequence
\revU{$\{\pi_{\vecX_*^{(N)}}\}_{N \geq 1} \subseteq \mathbb{P}(E^1)$}
is tight.
\end{proposition}

At this point, via Prohorov's theorem
\cite[Theorem 6.2, page 37]{Billingsley68}, we deduce that the
sequence $(\pi_{\vecX_*^{(N)}})_{N \geq 1} \subseteq \mathbb{P}(E^1)$
possesses converging subsequences. 
\revU{This convergence can be established through a line of argument
originally developed in~\cite{Whitt85} (where the state space is finite
and thus tightness automatic) and later extended in~\cite{GZ06}.}
\revT{Specifically, given any of the converging subsequences
$\{\pi_{\vecX_*^{(N_k)}}\}_{k \geq 1}$, denote by $\pi_{\vecX_*}$ its
limit.
The elements of the sequence $\{\pi_{\vecX_*^{(N_k)}}\}_{k \geq 1}$ may
be interpreted as a sequence of distributions from which the initial
condition of the initial-value problem~\eqref{eqn_MFDiffEqnGen_com}
is sampled.
Observe that, as $N_k \rightarrow \infty$, the network behavior is
governed by the solution of the mean-field limit and thus, given the
randomized initial condition, the population process evolves
deterministically as $\vecx(t)$.
So as to be the limit of a subsequence of stationary distributions,
$\pi_{\vecX_*}$ needs to be an invariant distribution for the
initial-value problem \eqref{eqn_MFDiffEqnGen_com}. 
\revB{Observe that $\delta_{\vecx^*}$ is an invariant distribution,
and moreover, given the results of Subsection~\ref{Sec_Compl_GS}, 
here are no others.}
In fact, because of global stability, any solution $\vecx(t)$
with finite initial mass converges to~$\vecx^*$ as $t\rightarrow\infty$,
and due to Proposition~\ref{Prop_Comp_Tightness}, the initial condition
sampled from $\pi_{\vecX_*}$ has finite mass with probability~$1$.
Hence, we conclude that $\vecX_*^{(N)} \Rightarrow \vecx^*$
as $N \rightarrow \infty$, completing the proof
of Theorem~\ref{Theo_Comp_IL}.

\subsection{Performance measures}
\label{Sec_Tag}

\revU{We now leverage Theorem~\ref{Theo_Comp_IL} for the population
process to obtain various performance measures of interest.
In particular, we sketch the proofs of
Theorems~\ref{Theo_QWSNcWeakConvergence}
and~\ref{Theo_QWSNcExpConvergence}, yielding approximations for the
stationary distribution and expectation of the queue length of a node
as well as the sojourn time and waiting time of a packet in the system.

\label{Sec_Tag_Distr}

Consider a tagged class-$c$ node.
In order to obtain an approximation for the stationary performance
measures, we rely on the symmetry of the class-$c$ nodes and leverage
the results established in the previous subsection.}
In particular, for $N$ sufficiently large, we aim to use the 
following approximation
\begin{equation}
\label{eqn_ConvDistrQRelation1}
\lim_{t \rightarrow \infty} \mathbb{P} \{Q_{c,1}^{(N)}(t) = k\} \approx
\lim_{N \rightarrow \infty}\lim_{t \rightarrow \infty} \mathbb{P} \{Q_{c,1}^{(N)}(t) = k\}, 
\end{equation}
\revU{where the right hand side is tractable
due to Theorem~\ref{Theo_Comp_IL}.
Specifically, due to the symmetry of nodes within the same class,
we have that}
\[
\lim_{N \rightarrow \infty} \lim_{t \rightarrow \infty}
\mathbb{P}\{Q_{c,1}^{(N)}(t) = k\} = x_{c,k}^* = (1 - \xi_c)\xi_c^k,
\]
 
and therefore
\[
Q_c^{(N)} \Rightarrow \bar{Q}_c, \qquad
\bar{Q}_c \sim \text{Geo}(\xi_c), \qquad c \in \mathcal{C}.
\]

Proceeding to the waiting-time distribution, denote by $\phi_{X}(\cdot)$
and $F_{X}(\cdot)$ the Laplace-Stieltjes transform and the probability
generating function of a random variable~$X$, and define the random
variable $\bar W^{(N)}_{c} = \frac{\lambda_c}{N_c} W_{c}^{(N)}$.
\revU{It follows from \ref{PseudoCons1_Upper2} that the sequence $\bar W^{(N)}_{c} $
is tight as a consequence of Markov's inequality. We now determine the limit for an
arbitrary sequence. }
Observe that the Laplace-Stieltjes transform $\phi_{W^{(N)}_c}$
and the probability generating function $F_{Q_{c}^{(N)}}(z)$ 
satisfy the distributional form of Little's law, i.e.
\begin{equation}
\phi_{\bar W^{(N)}_c}(z) = \phi_{\frac{\lambda_c}{N_c} W^{(N)}_c}(z) =
F_{Q_c^{(N)}}(1-z), \qquad z \in [0,1],
\label{eqn_DistLittle}
\end{equation}
see~\cite{BN95}.

The convergence of $F_{Q_c^{(N)}}(1-z)$
to $\frac{1-\xi_c}{1 - \xi_c(1-z)}$ can then be leveraged to prove
\[
\frac{\lambda_c}{N_c} W_{c}^{(N)} \Rightarrow \bar{W}_c, \qquad
\bar{W}_c  \sim \text{Exp}\big(\frac{1-\xi_c}{\xi_c}\big), \qquad
c \in \mathcal{C}.
\]

Finally turning to the sojourn time distribution, we define the random
variable $\bar S_c^{(N)}= \frac{\lambda_c}{N_c} S_c^{(N)}$,
and observe that
\[
\bar S_c^{(N)} \sim \bar W_{c}^{(N)} + \frac{\lambda_c}{N_c} U_c^{(N)},
\qquad c \in \mathcal{C},
\]
where $U_c^{(N)}\sim \text{Exp}(\mu_c^{(N)})\sim \text{Exp}(\mu_c)$
denotes the transmission time of a class-$c$ packet.
Therefore
\[
\frac{\lambda_c}{N_c} S_c^{(N)} \Rightarrow \bar{S}_c, \qquad
\bar{S}_c \sim \text{Exp}\big(\frac{1-\xi_c}{\xi_c}\big), \qquad
c \in \mathcal{C}.
\]

\revU{Above we sketched the proof
of Theorem~\ref{Theo_QWSNcWeakConvergence}, furnishing the limiting
distributions of the stationary queue length, waiting-time,
and sojourn time distributions of the various classes.
Such convergence results do not allow any immediate conclusions
for the corresponding expectations as stated
in Theorem~\ref{Theo_QWSNcExpConvergence},
but in this case these can be established by once again exploiting
the equivalence with a 1-limited polling system with random routing
described in Section~\ref{Sec_Compl_PR} and specifically the
pseudo-conservation law \eqref{PseudoCons1}}.

\section{Numerical experiments}
\label{Sec_Num}

In this section we perform numerical experiments in order to
complement the analytical results obtained in the previous sections.
In particular, we \revU{discuss an example with a non-complete
interference graph and no symmetries.
We visually display the result proved in Theorem~\ref{Theo_MFGen},
and consider both a scenario in which a fixed point exists and one
in which it does not.
For additional details on the quality of the approximations derived in Section \ref{Sec_Compl}, we refer to the analysis in \cite[Section 4.6]{Cecchi18} where it is shown that these prove to be remarkably accurate even in moderately large networks.}

\begin{figure}
\begin{center}
\includegraphics[width=0.5\textwidth]{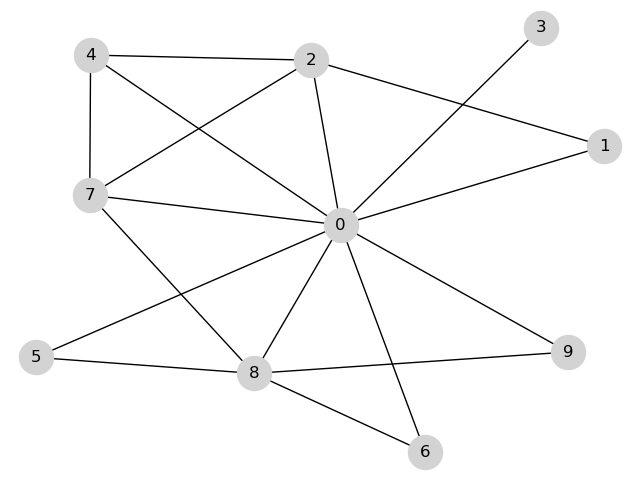}
\caption{10-class interference graph.
\label{fig_IntGraph10Nodes}}
\end{center}
\end{figure}

\revU{
Throughout this section we focus on the 10-class network displayed
in Figure~\ref{fig_IntGraph10Nodes} with the following parameters:
\begin{align*}
\lambda_c=0.4,\qquad &\text{for } c = 1,\ldots,9\\
\nu_c=\mu_c=3, \qquad &\text{for } c = 0,\ldots,9
\end{align*}
and an equal number of nodes in each class, i.e., $p_c = 0.1$
for $c=0,\ldots,9$.
We let $\lambda_0$ vary and consider two scenarios.

In Scenario~1, we set $\lambda_0 = 0.25$, and we verify that
$\vecxi = \veceta(\vecrho) < \vece$ to ensure that a unique
\revT{fixed} point $\vecx^*$ of~\eqref{eqn_MFDiffEqnGen} exists,
by virtue of Theorem~\ref{Theo_EquilibChar}.
In order to solve~\eqref{eqn_FormulaXi}, we use the numerical
algorithm described in~\cite{CBLW16}, which is an adaptation of the
algorithm presented in~\cite{VJLB11}, and obtain
\[
\vecxi \approx (0.478,\: 0.205,\: 0.311,\: 0.170,\: 0.258,\: 0.205,\: 0.205,\: 0.311,\: 0.359,\: 0.205).
\]
Note that $\vecxi < \vece$, and Theorem~\ref{Theo_EquilibChar} thus
yields that
\[
\vecx^* = (x^*_{c,n}), \qquad x^*_{c,n}=(1-\xi_c)\xi_c^n,
\]
is the unique \revT{fixed} point. 

In Scenario~2, we double the arrival rate at node~0, i.e.,
$\lambda_0 = 0.5$, and observe that \eqref{eqn_MFDiffEqnGen} has no
fixed points.
In fact, by means of the numerical algorithm described in~\cite{CBLW16}
we obtain that
\[
\vecxi \approx (1.317,\: 0.235,\: 0.380,\: 0.190,\: 0.308,\: 0.235,\: 0.235,\: 0.380,\: 0.443,\: 0.235).
\]
In this case $\xi_0 > 1$, and in view of Theorem~\ref{Theo_EquilibChar},
no \revT{fixed} point for~\eqref{eqn_MFDiffEqnGen} exists. \\

\noindent
{\it Convergence of the population process to the solution
of~\eqref{eqn_MFDiffEqnGen}.}
We first visually illustrate the result stated in Theorem~\ref{Theo_MFGen},
i.e, the convergence of the properly scaled population process
$\vecX^{(N)}(t)$ to the solution $\vecx(t)$ of~\eqref{eqn_MFDiffEqnGen}
for arbitrary interference graphs. 
In Figure~\ref{fig_MFGenConverge_10_conv} we display sample paths
of the scaled population process and compare these with the solution
$\vecx(t)$ of~\eqref{eqn_MFDiffEqnGen}.
We show the results for both Scenarios~1 and~2, and confirm that
Theorem~\ref{Theo_MFGen} holds independently of the existence of the
fixed point.
For both cases we let the number of nodes per class be $100$ or $1000$.
Note that the oscillations around $\vecx(t)$ decrease in magnitude
as $N$ increases. \\

\begin{figure}
\begin{center}
\includegraphics[width=0.44\textwidth]{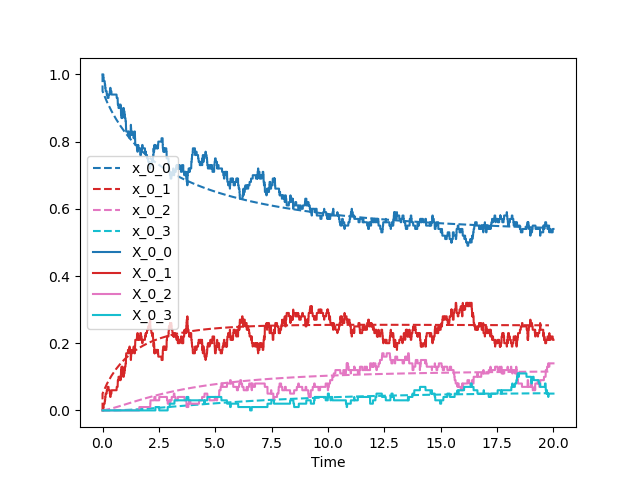}
\includegraphics[width=0.44\textwidth]{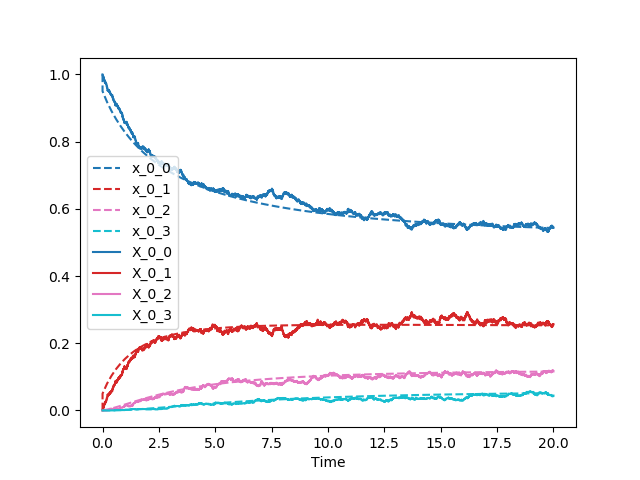}
\includegraphics[width=0.44\textwidth]{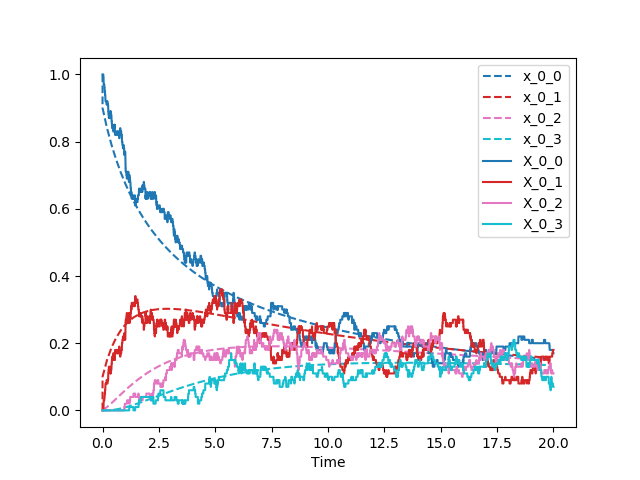}
\includegraphics[width=0.44\textwidth]{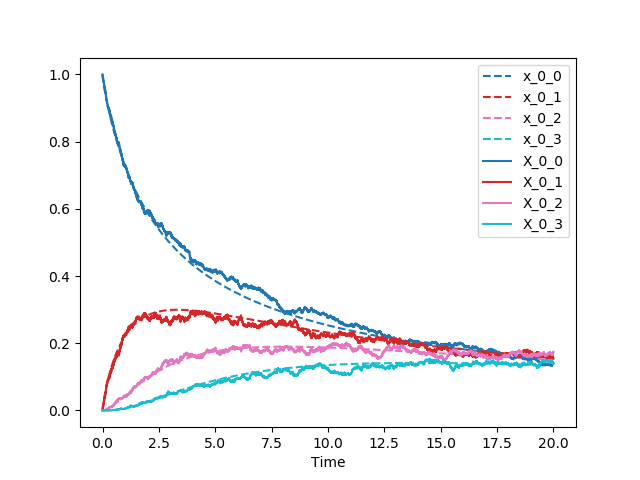}
\caption{Plot of the sample path $X^{(N)}(Nt)$ and of the solution
$x_{c,n}(t)$ of~\eqref{eqn_MFDiffEqnGen} for $c=0$ and $n=0,1,2,3$.
Both for Scenario~1 (top) and Scenario~2 (bottom) with $N_c = 100$
(left) and $N_c=1000$ (right).
\label{fig_MFGenConverge_10_conv}}
\end{center}
\end{figure}

\noindent
{\it Scenario~1. Convergence to the fixed point.}
For Scenario~1 we derived a unique fixed point~$\vecx^*$
of~\eqref{eqn_MFDiffEqnGen}.
In Figure~\ref{fig_MFGenConverge_10} we display a numerical solution
of~\eqref{eqn_MFDiffEqnGen} with empty initial condition,
and observe that $\vecx(t) \rightarrow \vecx^*$ as $t \to \infty$.
This supports the hypothesis that the solution $\vecx(t)$ converges
to~$\vecx^*$ even when the interference graph is not complete. \\

\begin{figure}
\begin{center}
\includegraphics[width=0.44\textwidth]{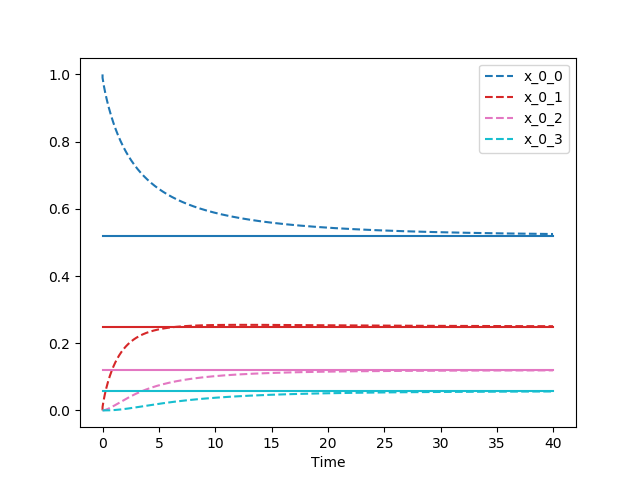}
\includegraphics[width=0.44\textwidth]{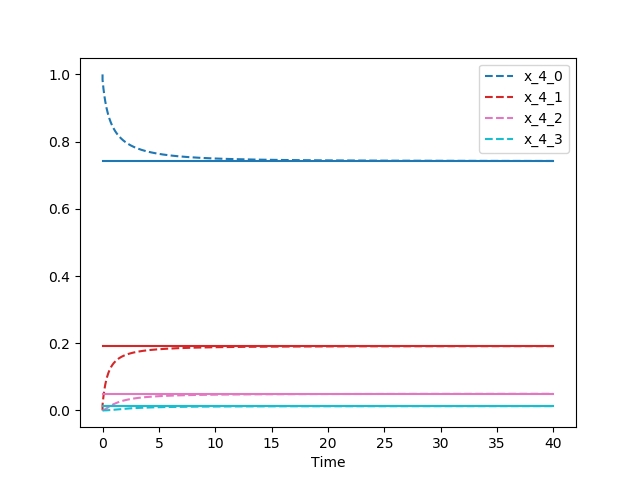}
\caption{Scenario 1. Plot of the solution $x_{c,n}(t)$
of~\eqref{eqn_MFDiffEqnGen} for $c=0,4$ and $n=0,1,2,3$.
\label{fig_MFGenConverge_10}}
\end{center}
\end{figure}

\noindent
{\it Scenario~2. Explosion of the population process.} 
For Scenario~2 we established that there exists no fixed point
of~\eqref{eqn_MFDiffEqnGen}.
In this case, we observe a dichotomy in behavior.
Some classes explode, i.e., the number of backlogged packets in the
buffer of every node grows without bound.
Other classes remain stable, but the stationary buffer content
distribution is no longer geometric with parameters given by the
activity factors $\vecxi$.   
In Figure~\ref{fig_MFGenConverge_10_over} we display a numerical
solution of~\eqref{eqn_MFDiffEqnGen} with empty initial condition.
Observe that for class~$0$, the solution is non-monotone and converges
to~$0$.
In contrast, for class~$4$, the solution converges to a fixed point
which is however no longer characterized by~$\vecxi$ (even if very close).

\begin{figure}
\begin{center}
\includegraphics[width=0.44\textwidth]{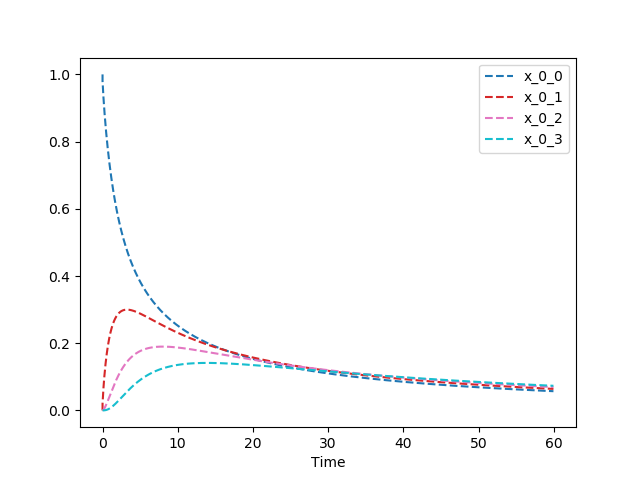}
\includegraphics[width=0.44\textwidth]{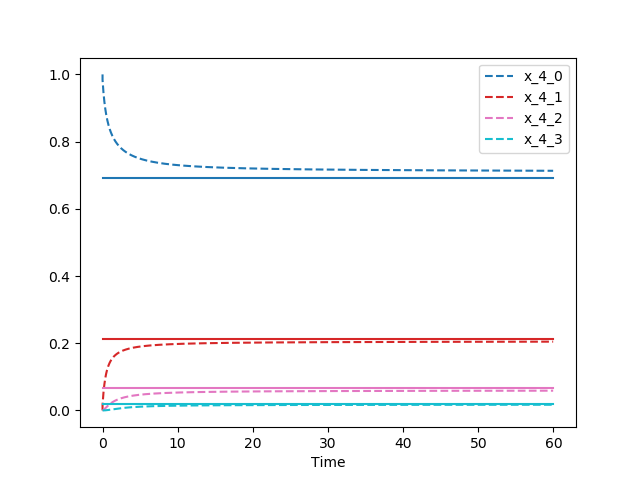}
\caption{Scenario 2. Plot of the solution $x_{c,n}(t)$
of \eqref{eqn_MFDiffEqnGen} for $c=0,4$ and $n=0,1,2,3$.
\label{fig_MFGenConverge_10_over}}
\end{center}
\end{figure}
}

\section{Conclusion}
\label{Sec_Conclusion}

We analyzed distributed random-access networks in a many-sources regime
with buffer dynamics.
The difficulty in the analysis is the complex interaction between the
medium activity state and the buffer content process.
However, in the mean-field regime, these processes are shown to decouple
and a tractable initial-value problem describing the buffer content
process is obtained. 
The methodological framework that we developed is fairly generic,
and expected to be applicable to a broader range of problems
where a similar decoupling occurs. 

For the case of a complete interference graph the stationary buffer
content and the packet waiting and sojourn time are shown to converge
in distribution and expectation to random variables whose parameters
are explicitly given by the unique \revT{fixed} point of the 
initial-value problem.
These results are obtained via an interchange of limits approach which
ensures that the resulting approximations are asymptotically exact.
Numerical experiments suggest that these results extend to general
interference graphs, but a rigorous proof remains as a challenging
problem for further research.

\revR{The theoretical machinery for the derivation of mean-field limits
developed in Section~\ref{Sec_Gen} is highly flexible and may be
applied to a wide range of models. 
The basic model described in Section~\ref{Sec_Model} can be extended,
and the associated mean-field limits may be obtained along similar lines.
For instance, with minimal additional effort, mean-field limits can be
obtained for models with queue-based back-off rates and finite-capacity
buffers.
The time-scale separation between the activity process and the
population process plays a key role once again.
However, the expression for the stationary distribution of the activity
process with a fixed population is model-dependent, and this is
reflected in the corresponding limiting initial-value problems
which differ only in the expression for $\pi_{\vecx}(\Omega_{-c})$,
see~\cite{FabioThesis18} for further details.}

\appendix

\section{Extended proofs} 
\label{App_Proofs}

\subsection{Proof of Lemma \ref{prop_dZ}}
\label{App_prop_dZ}

\revR{We already observed that} the process $\vecXI^{(N)}$ is a multi-parameter martingale
adapted to a filtration ${\bf F}$. Since $C$ is finite and $\frac{\abs{ \Xi^{(N)}_{c,n}(t)}}{1 + \abs{\Xi^{(N)}_{c,n}(t)}} \leq 1$ for every
$t \geq 0$, 
there exists \revB{$K_\eta$ such that
$$
\sumc \sum_{n > K_\eta}  2^{-n}\frac{\abs{ \Xi^{(N)}_{c,n}(t)}}{1 + \abs{\Xi^{(N)}_{c,n}(t)}} < \eta/2.
$$}
For each $c \in \CC, n \leq K_\eta$, $\Xi^{(N)}_{c,n}(t)$ is a ${\cal L}^2$ ${\bf F}$ martingale so
that $\lb \Xi^{(N)}(t) \rb^2$ is a ${\cal L}^1$ ${\bf F}$ submartingale. Applying Doob's submartingale inequality with $p=1$ 
(see \cite{RY13}),
we obtain
\begin{equation*}
\pr{ \sup_{t \in [0,T]} \abs{ \Xi^{(N)}_{c,n}(t)}^2  > \varepsilon} \leq \frac{\expect{ \abs{ \Xi^{(N)}_{c,n}(T)}^2 }}{\epsilon},\qquad c \in \CC, n \leq K_\eta.
\end{equation*}
However $\expect{ \abs{ \Xi^{(N)}_{c,n}(T)}^2 } = O(1/N)$ from which the result follows. 

\revR{
\subsection{Lipschitz continuity of the function $H(\cdot)$}
\label{App_LipContH}

For simplicity, denote by $\vecx^{(c)} = (x_{c,n})_{n \in \NatZero}$ and we know that $\vecx^{(c)} \in \chi$
which is closed and compact under the poduct topology. We will leverage the following lemma, whose proof 
is omitted.
\begin{lemma}
\label{Lemma_LipContH}
 Given $\vecf,\vecg: E  \rightarrow \prod_{\cCC} \mathbb{R}^{\infty} $
 Lipschitz continuous functions, then so is $\vech$, where $\vech(\vecx) = \vecf(\vecx) + \vecg(\vecx)$. Given
 $f,g: E  \rightarrow \mathbb{R}$ Lipschitz continuous and bounded functions, then 
 so is $h$, where $h(\vecx):=f(\vecx)g(\vecx)$.
\end{lemma}
Let us recall that
\[
 \rho\big( (\vecx^{(1)},\ldots,\vecx^{(C)}) , (\vecy^{(1)},\ldots,\vecy^{(C)})\big) = \sum_{\cCC} \rho_1\big(\vecx^{(c)} , \vecy^{(c)}\big).
\]

Note that the first part of function $H(\vecx)$, i.e., 
\[
 \sum_{\cCC} \frac{1}{p_c}  \lambda_c \sum_{n \in \NatZero} x_{c,n} \vece_n^{n+1},
\]
is clearly Lipschitz. Let us now focus on the second part, for simplicity we neglect the constants $1/p_c$ and $\nu_c$.
Consider the term
\[
 \pi_{\vecx_0}(\Omega_{-c})x_{c,n}
\]
for some $\cCC$. Clearly $x_{c,n}$ is Lipschitz continuous and bounded for every $\vecx^{(c)} \in \chi$, as for the term $\pi_{\vecx_0}(\Omega_{-c})$
it is clearly bounded and we now show that it is Lipschitz continuous in $\vecx$. Note that, we can write $\pi_{\vecx_0}(\Omega_{-c})$ as
\[
 \pi_{\vecx_0}(\Omega_{-c}) = \frac{a^{(c)}(x_0^{(1)},\ldots,x_0^{(C)})}{b^{(c)}(x_0^{(1)},\ldots,x_0^{(C)})},
\]
where both $a^{(c)}$ and $b^{(c)}$ are polynomials in the variables $\{x^{(c)}_0\}_{\cCC}$ such that
\[
 a^{(c)}(x_0^{(1)},\ldots,x_0^{(C)}) \geq 0, \qquad b^{(c)}(x_0^{(1)},\ldots,x_0^{(C)}) \geq 1,
\]
for every $(\vecx^{(1)},\ldots,\vecx^{(C)}) \in E$. Hence, it holds that
\[
 \frac{\partial a^{(c)}}{\partial x_0^{(d)}} = \frac{\partial \pi_{\vecx_0}(\Omega_{-c})}{\partial x_0^{(d)}}b^{(c)} + \pi_{\vecx_0}(\Omega_{-c})\frac{\partial b^{(c)}}{\partial x_0^{(d)}}. 
\]
Observe that $\frac{\partial a^{(c)}}{\partial x_0^{(d)}}$ and $\frac{\partial b^{(c)}}{\partial x_0^{(d)}}$ are both bounded and continuous, and therefore
$\frac{\partial \pi_{\vecx_0}(\Omega_{-c})}{\partial x_0^{(d)}}$ is bounded as well for every $c,d \in \CC$. Since $C$ is finite, there exists $L_{\pi} < \infty$
such that
\[
 \big|\frac{\partial \pi_{\vecx_0}(\Omega_{-c})}{\partial x_0^{(d)}} \big| \leq L_{\pi}, \qquad \forall c,d \in \CC.
\]
Hence, we have that 
\begin{align*}
 \big|\pi_{\vecx_0}(\Omega_{-c}) - \pi_{\vecy_0}(\Omega_{-c}) \big| &\leq L_{\pi} \sum_{d \in \CC}|x_0^{(d)} - y_0^{(d)}| \\
 &\leq L_{\pi} \rho\big( (\vecx^{(1)},\ldots,\vecx^{(C)}) , (\vecy^{(1)},\ldots,\vecy^{(C)})\big),
\end{align*}
yielding that $\pi_{\vecx_0}(\Omega_{-c})$ is Lipschitz continuous in $\vecx \in  E$.
The above results, together with Lemma \ref{Lemma_LipContH}, suffice to prove the Lipschitz continuity of function $H(\vecx)$.
}

\subsection{Proof of Corollary~\ref{alphacorollary}}
\label{proofalphacorollary}

Thanks to the fundamental theorem of calculus applied
to~\eqref{FLCumulativeTransActivProc}, it holds that
\begin{align}
0 = \nu_c (1- x_{c,0}(t)) \sum_{\omega \in \Omega_{-c}} \dtalp_{\omega}(t) -
\mu_c \sum_{\omega \in \Omega_{+c}} \dtalp_\omega(t), \qquad \text{a.e.}
\label{YmartingaleLimit}
\end{align}
Hence, at every time~$t$, we have to look for solutions
$(z^t_{\omega})_{\omega \in \Omega}$ of
\begin{equation}
\nu_c (1- x_{c,0}(t)) \sum_{\omega \in \Omega_{-c}} z_\omega^t =
\mu_c \sum_{\omega \in \Omega_{+c}} z_\omega^t, \quad \forall c \in \mathcal{C}, 
\label{eqn_CSMASaturStat}
\end{equation}
where in view of~\eqref{RelationCumOccAct}
$\sum_{\omega \in \Omega} z_\omega^t = 1$ and $z_\omega^t \geq 0$.
Observe that, $z_\omega^t$ corresponds to the stationary fraction
of time spent by the activity process in state $\omega \in \Omega$
in the saturated network where node~$c$ transmits at rate $\mu_c$
and back-offs at rate $\nu_c(1 - x_{c,0}(t))$, i.e.,
\[
\dot\alpha_\omega(t) = z_\omega^t = \pi(\omega; \vece - \vecx_{0}(t)) =
\pi_{\vecx_0(t)}(\omega).
\]
In particular, existence and uniqueness of the stationary distribution
$\pi_{\vecx_0(t)}(\omega)$ is ensured whenever $\vecx_0(t) \leq \vece$.
Hence, $\dtalp_{\omega}(t) = z^t_{\omega} $ is well-defined and, since an
increasing unit Lipschitz function starting from~$0$ is determined by its
derivatives, the equality holds for all~$t$, determining $\alpha_\omega$. 
Observe that $\alpha_\omega$ exists everywhere since it coincides
with the integral of a continuous function.

\subsection{Proof of Proposition~\ref{Prop_Comp_Tightness}}
\label{App_Prop_Comp_Tightness}

\revT{Given the sequence of stationary distributions of the population
process $\{\pi_{\vecX_*^{(N)}}\}_{N \geq 1}$, we will show that for every
$\epsilon > 0$ there exist a compact set $\mathcal{K}_\epsilon \subset E^1$
and $N_\epsilon > 0$ such that
\begin{equation}
\mathbb{P}\big\{\vecX^{(N)} \in \mathcal{K}_\epsilon\big\} > 1 - \epsilon,
\qquad \forall \: N>N_\epsilon.
\label{T04_eqn_TightnessCrit}
\end{equation}
In particular, we will show that the above relation holds for a compact
set $\mathcal{K}_{\epsilon}$ defined as
\[
\mathcal{K}_\epsilon =
\big\{\vecx \in E^1: \sum_{\cCC} m_c(\vecx) \leq K_\epsilon \big\}.
\]
 
Let us now show that $\mathcal{K}_{\epsilon}$ is compact.
Note that $\mathcal{K}_\epsilon \subseteq E$, which is compact.
Hence, it is sufficient to prove that $\mathcal{K}_\epsilon$ is closed.
Observe that any sequence $\vecx_n$ in $\mathcal{K}_\epsilon$ is also
a sequence in~$E$.
Let $\vecx$ be a limit point of this sequence which is necessarily in~$E$.
But clearly
\[
\sum_{\cCC} m_c({\vecx}_n) \leq K_\epsilon, \qquad \forall n
\]
and $\vecx_n$ is a vector of probabilities on $\NatZero$
with $C$~components.
It follows that each component sequence is tight and therefore the
limit must be in~$E^1$.
Now let $\mu_c$ be the expectation of the limit (possibly infinite).
Then, from \cite[Theorem 5.3, p.~32]{Billingsley68}, we know that
\[
\mu_c \leq \liminf m_c({\bf x}_n) \leq K_\epsilon.
\]
Since $\sum_c \mu_c \leq \sum_c \liminf m_c({\vecx}_n) \leq
\liminf \sum_c m_c({\vecx}_n) \leq K_\epsilon$
the required bound on the sum expectations holds and so the limit is
in $\mathcal{K}_\epsilon$, which is therefore closed.
 
Note that
\[
\mathbb{P}\big\{\vecX_*^{(N)} \notin \mathcal{K}_\epsilon\big\} =
\mathbb{P}\big\{\sum_{\cCC} m_c(\vecX_*^{(N)}) > K_\epsilon\big\} \leq
\frac{\sum_{\cCC } \mathbb{E}\Big[m_c(\vecX_*^{(N)})\Big]}{K_\epsilon},
\]
due to Markov's inequality. 

Consider class $\cCC$.
Let $Q_{c,k}^{(N)}$ be the random variable denoting the number of packets
waiting in the buffer of the $k$-th class-$c$ node in stationarity.
Note that $Q_{c,k}^{(N)} = Q_c^{(N)}$ for every $k \in \{1,\ldots,N_c\}$,
due to the exchangeability of the nodes within the same class,
and observe that 
\begin{align*}
\mathbb{E}[Q_c^{(N)}] &= \sum_{\vecq \in \NatZero^N}
\pi_{\vecQ^{(N)}}(\vecq) \frac{1}{N_c} \sum_{k=1}^{N_c} q_{c,k} \\ 
&= \sum_{\vecq \in \NatZero^N} \pi_{\vecQ^{(N)}}(\vecq) \frac{1}{N_c} \sum_{k=1}^{N_c} \sum_{n=1}^{\infty} n \mathbbm{1}\{q_{c,k} = n\} \\
&= \sum_{\vecx \in E^1} \pi_{\vecX_*^{(N)}}(\vecx)
\sum_{n=1}^{\infty} n x_{c,n} = \mathbb{E}\big[m_c(\vecX_*^{(N)})\big], 
\end{align*}
where the second equality is due to~\eqref{T04_eqn_RelQX}.
Hence,
\begin{equation}
\mathbb{P}\big\{\vecX_*^{(N)} \notin \mathcal{K}_\epsilon\big\} \leq
\frac{\sumc\mathbb{E}[Q_c^{(N)}]}{K_{\epsilon}}.
\label{ProofTightness1}
\end{equation}

For every class $\cCC$, by neglecting the time periods in which
class-$c$ nodes cannot activiate due to activity of their neighbors,
we obtain that
\[
\big(\frac{1}{\mu_c} + \frac{N_c}{\nu_c}\big) \mathbb{E}[Q^{(N)}_c] \leq
\mathbb{E}[W^{(N)}_c], 
\]
and therefore
\begin{equation}
\sumc \mathbb{E}[Q^{(N)}_c] \leq
\sumc \big(\frac{1}{\mu_c}+\frac{N_c}{\nu_c} \big)^{-1} \mathbb{E}[W^{(N)}_c] \leq
\sumc \frac{\nu_c}{N_c} \mathbb{E}[W^{(N)}_c].
\label{ProofTightness2}
\end{equation}
Thus, it suffices to show that $\mathbb{E}[ W^{(N)}_c]$ grows at most
linearly in~$N$ as $N \rightarrow \infty$.

Combining~\eqref{ProofTightness1} and~\eqref{ProofTightness2}, we obtain 
\[
\mathbb{P}\big\{\vecX_*^{(N)} \notin \mathcal{K}_\epsilon\big\} \leq
\frac{1}{K_\epsilon} \sumc \frac{\nu_c}{N_c} \mathbb{E}[W_c^{(N)}].
\]
In view of~\eqref{PseudoCons1_Upper2}, for any $\epsilon > 0$ we can
therefore pick $K_\epsilon$ such that \eqref{T04_eqn_TightnessCrit}
is asymptotically satisfied.
}

\end{document}